\documentclass[reqno, 15pt]{amsart}
\usepackage{latexsym}
\usepackage{amsmath}
\usepackage{amssymb}
\usepackage{amsthm}
\usepackage{amscd}
\usepackage{graphicx}
\usepackage{xcolor}
\usepackage[colorlinks,citecolor=blue]{hyperref}
\usepackage{tikz}
\usepackage{tikz-cd}
\usepackage{caption} 
\usepackage[a4paper,top=3cm,bottom=3cm,left=3cm,right=3cm]{geometry}
\usepackage{cite}
\usepackage{enumitem}
\usepackage{tabularx}
\usepackage{multirow}
\usepackage{multicol}
\usepackage{longtable}
\usepackage{makecell}

\usepackage{float} 
\newtheorem{theorem}{Theorem}[section]

\newtheorem{lemma}[theorem]{Lemma}
\newtheorem{proposition}[theorem]{Proposition}

\theoremstyle{definition}

\theoremstyle{remark}
\newtheorem{remark}[theorem]{Remark}

\numberwithin{equation}{section}

\DeclareMathOperator{\oh}{\mathcal{O}}

\begin{document}
	
\title[]{Uniform bundles on generalised Grassmannians}

\author{Xinyi Fang}
\address{Xinyi Fang, Department of Mathematics, Shanghai Normal University, Shanghai, 200234, PR China}
\thanks{1.Xinyi Fang is supported by innovation Action Plan (Basic research projects) of Science and Technology Commission of Shanghai Municipality (Grant No.21JC1401900)}
\email{xinyif@shnu.edu.cn}
\author{Duo Li*}
\address{Duo Li, Sun-Yat Sen University, School of Mathematics(Zhuhai), Zhuhai, Guangdong, 519082, PR China}
\thanks{2.Duo Li is supported by National Natural Science Foundation of China (Grant No. 12001547)}
\email{liduo5@mail.sysu.edu.cn}
\author{Yanjie Li}
\address{Yanjie Li, AMSS, Chinese Academy of Sciences, 55 ZhongGuanCun East Road, Beijing, 100190, China and  University of Chinese Academy of Sciences, Beijing, China}
\address{Sun-Yat Sen University, School of Mathematics(Zhuhai), China}
\email{liyj293@mail2.sysu.edu.cn}

\begin{abstract}
Let $E$ be a uniform bundle on an arbitrary generalised Grassmannian $X$ defined over $\mathbb{C}$. We show that  if the rank of $E$ is at most $e.d.(\mathrm{VMRT})$, then $E$  necessarily splits. For some generalised Grassmannians, we prove that the upper bounds $e.d.(\mathrm{VMRT})$ are optimal and classify all unsplit uniform bundles of minimal ranks.  Under some special assumptions, we show that  morphisms to some generalised flag varieties must be constant, which partially answered a conjecture of Kumar.
\end{abstract}

\maketitle

\setcounter{tocdepth}{1}
\tableofcontents


\section{Introduction}
In this article, we assume all the varieties are defined over $\mathbb{C}$. It is well known that every vector bundle on a rational curve $\mathbb{P}^1$ splits as a direct sum of line bundles. Hartshorne conjectures that any vector bundle of rank $2$ on $\mathbb{P}^n$ ($n\geq 7$) splits. This conjecture is still widely open. On the other hand, the splitting of uniform bundles on $\mathbb{P}^n$ of lower ranks is completely understood.  On a projective space $\mathbb{P}^n$, we consider a vector bundle $E$ and its restriction  $E|_{L}$ to any line $L\subseteq \mathbb{P}^n$. If the splitting type of   $E|_{L}$ is independent of the choice of $L$, we call $E$ a uniform bundle. Every uniform bundle on $\mathbb{P}^n$ whose  rank is  smaller than $n$ splits (cf. \cite[Main Theorem]{sato1976uniform}). The concept of uniform bundles can be generalised to the projective varieties swept by lines. Especially,  on a Grassmannian $Gr(k,n+1)(k\leq n+1-k)$, any uniform bundle whose  rank is smaller than $k$ splits (cf. \cite[Theorem 1]{guyot1985caracterisation}). On some other generalised Grassmannians, there are similar results, see \cite[Theorem 3.1]{munoz2012uniform}, \cite[Theorem 5.6]{Pan2015Tri} and \cite[Theorem 1.5]{du2021vector}.\\

On a projective space $\mathbb{P}^n$, up to twisting by a line bundle, an unsplit uniform bundle of rank $n$ must be isomorphic to $T_{\mathbb{P}^n}$ or $\Omega_{\mathbb{P}^n}$ (cf. \cite[Main Theorem]{EHS}). In \cite[Theorem 1]{guyot1985caracterisation},  Guyot classifies unsplit uniform bundles of rank $k$ on $Gr(k,n+1)(k\leq n+1-k)$. In \cite[Theorem 3.1]{munoz2012uniform}, the authors reprove Guyot's result and obtain a similar classification for the orthogonal isotropic Grassmannian $OG(n,2n+1)$. In \cite[Theorem 3.1]{fritzsche1983linear},  Fritzsche classifies unsplit uniform bundles of rank $2$ on $\mathbb{Q}^3$. The classification of unsplit uniform bundles of minimal ranks for other generalised Grassmanninas, even for quadrics $\mathbb{Q}^n(n\ge 5)$, is rarely known. \\

We aim to generalise the above results to arbitrary generalised Grassmannians and the 
 main theorem of our article is as follows.
\begin{theorem}\label{main1}
         $(1) $  Let $X$ be an arbitrary generalised Grassmannian $G/P$ and let $E$ be a uniform bundle of rank $r$ on $X$. If $r$ is at most $e.d.(\mathrm{VMRT})$,  the vector bundle $E$ splits as a direct sum of line bundles. \\
         
         $(2)$  Let $X$ be a generalised Grassmannian in Table \ref{Table 2}. There exists an
  unsplit uniform bundle $E_{\lambda}$ on $X$ of rank $e.d.(\mathrm{VMRT})+1$, where  $E_{\lambda}$ is the
    irreducible homogeneous bundle corresponding to some highest weight $\lambda$.\\

     $(3)$ Let $X$ be a generalised Grassmannian in Table \ref{Table 3}.
        Up to twisting by a line bundle,  every unsplit uniform bundle $E$ on $X$ of rank  $e.d.(\mathrm{VMRT})+1$  is an irreducible homogeneous bundle in Table \ref{Table 3} or its dual.

\end{theorem}
For the convenience of reading, we list some necessary information for Theorem \ref{main1}, $(1)$ in Table \ref{table} (we will explain the meaning of $a(X)$ and $e.d.(\mathrm{VMRT})$ in Section $3$).

\begin{small}
\begin{center}
   \begin{longtable}{|c|c|c|c|c|}
 \hline
 Type& $X$ & VMRT & $a(X)$& $e.d.$(VMRT)\\
\hline
 \multirow{3}{*}{$A_n(n\geq 1)$} & $\mathbb{P}^n$  & $\mathbb{P}^{n-1}$ & $n-1$& $n-1$\\
\cline{2-5}
 & $Gr(k,n+1)$  & \multirow{2}{*}{$\mathbb{P}^{k-1}\times \mathbb{P}^{n-k}$} & \multirow{2}{*}{$\min(k-1,n-k)$}& \multirow{2}{*}{$\min(k-1,n-k)$}\\
 & $(2\leq k \leq n-1)$ &&&\\
\hline
\multirow{5}{*}{$B_n(n\geq 2)$}& $\mathbb{Q}^{2n-1}$ & $\mathbb{Q}^{2n-3}$ & $2n-3$ & $2n-3$\\
\cline{2-5}
& $OG(k,2n+1)$ & \multirow{2}{*}{$\mathbb{P}^{k-1}\times \mathbb{Q}^{2(n-k)-1}$}&$\min$& $\min$\\
& $(2\leq k \leq n-2)$&& $(k-1,2n-2k-1)$ & $(k-1,2n-2k-1)$\\
\cline{2-5}
& $OG(n-1,2n+1)$ & $\mathbb{P}^{n-2}\times \mathbb{P}^1$ & $1$ & $1$\\
\cline{2-5}
& $OG(n,2n+1)$& $Gr(n-1,n+1)$ & $2n-2$ & $n$\\
\hline
\multirow{3}{*}{$C_n(n\geq 3)$}& $SG(k,2n)$ & \multirow{2}{*}{$\mathbb{P}_{\mathbb{P}^{k-1}}(\oh(-2)\oplus \oh(-1)^{2n-2k})$}& \multirow{2}{*}{$\min(k-1,2n-2k)$} & \multirow{2}{*}{$\min(k-1,2n-2k)$} \\
& $(2\leq k \leq n-1)$ &&&\\
\cline{2-5}
& $LG(n,2n)$ & $\mathbb{P}^{n-1}$ & $n-1$ & $n-1$\\
\hline
\multirow{5}{*}{$D_n(n\geq4)$} & $\mathbb{Q}^{2n-2}$ & $\mathbb{Q}^{2n-4}$ & $2n-4$ & $2n-5$\\ 
\cline{2-5}
& $OG(k,2n)$ & \multirow{2}{*}{$\mathbb{P}^{k-1}\times \mathbb{Q}^{2(n-k)-2}$}& $\min$ & $\min$\\ 
& $(2\leq k \leq n-3)$ & & $(k-1,2n-2k-2)$ & $(k-1,2n-2k-3)$\\
\cline{2-5}
& $OG(n-2,2n)$ & $\mathbb{P}^{n-3}\times \mathbb{P}^{1}\times \mathbb{P}^1$ & $1$ & $1$\\
\cline{2-5}
& $OG(n,2n)$ & \multirow{2}{*}{$Gr(n-2,n)$} & \multirow{2}{*}{$2n-4$} & \multirow{2}{*}{$n-1$}\\
& $(\cong OG(n-1,2n))$ & & &\\
\hline
\multirow{4}{*}{$E_6$} & $E_6/P_{1}(\cong E_6/P_6)$ & $OG(5,10)$& $10$ & $7$\\
\cline{2-5}
& $E_6/P_2$ & $Gr(3,6)$ & $9$ & $5$\\
\cline{2-5}
& $E_6/P_3(\cong E_6/P_5)$ & $Gr(2,5)\times \mathbb{P}^1$ & $1$ & $1$\\
\cline{2-5}
& $E_6/P_4$ & $\mathbb{P}^2\times \mathbb{P}^2 \times \mathbb{P}^1$ & $1$ & $1$ \\
\hline
\multirow{7}{*}{$E_7$} & $E_7/P_1$ & $OG(6,12)$ & $15$ & $9$\\
\cline{2-5}
& $E_7/P_2$ & $Gr(3,7)$ & $12$ & $6$\\
\cline{2-5}
& $E_7/P_3$ & $Gr(2,6)\times \mathbb{P}^1$ & $1$ & $1$\\
\cline{2-5}
& $E_7/P_4$ & $\mathbb{P}^3\times \mathbb{P}^2 \times \mathbb{P}^1$ & $1$ & $1$\\
\cline{2-5}
& $E_7/P_5$ & $Gr(2,5)\times \mathbb{P}^2$ & $2$ & $2$\\
\cline{2-5}
& $E_7/P_6$ & $OG(5,10)\times \mathbb{P}^1$ & $1$ & $1$\\
\cline{2-5}
& $E_7/P_7$ & $E_6/P_6$ &16 & $12$\\
\hline
\multirow{8}{*}{$E_8$}& $E_8/P_1$ & $OG(7,14)$ & $21$ & $11$\\
\cline{2-5}
& $E_8/P_2$ &$Gr(3,8)$ & $15$ & $7$\\
\cline{2-5}
& $E_8/P_3$ & $Gr(2,7)\times \mathbb{P}^1$ & $1$ & $1$\\
\cline{2-5}
& $E_8/P_4$ & $\mathbb{P}^4\times \mathbb{P}^2 \times \mathbb{P}^1$ & $1$ & $1$\\
\cline{2-5}
& $E_8/P_5$ & $Gr(2,5)\times \mathbb{P}^3$ & $3$ & $3$\\
\cline{2-5}
& $E_8/P_6$ & $OG(5,10)\times \mathbb{P}^2$ & $2$ & $2$\\
\cline{2-5}
& $E_8/P_7$ & $E_6/P_6\times \mathbb{P}^1$ & $1$ & $1$\\
\cline{2-5}
& $E_8/P_8$ & $E_7/P_7$ &$27$ & $19$\\
\hline
\multirow{4}{*}{$F_4$} & $F_4/P_1$ & $LG(3,6)$ &$6$ & $5$\\
\cline{2-5}
& $F_4/P_2$ & $\mathbb{P}^2\times \mathbb{P}^1$ & $1$ & $1$\\
\cline{2-5}
& $F_4/P_3$ & $\mathbb{Q}^4$-bundle over $\mathbb{P}^1$ & $1$ & $1$\\
\cline{2-5}
& \multirow{2}{*}{$F_4/P_4$} &  hyperplane section of the & \multirow{2}{*}{$9$} & \multirow{2}{*}{6}\\
&&10-dim Spinor variety&  &\\
\hline
\multirow{2}{*}{$G_2$} & $G_2/P_1$ &$\mathbb{Q}^3$ & $3$ & $3$\\
\cline{2-5}
& $G_2/P_2$ & $\mathbb{P}^1$ & $1$ & $1$\\
\hline
\caption{}
   \label{table}\\
\end{longtable}
\end{center}
\end{small}

For the optimal upper bounds, we refer to Table \ref{Table 2}.
\begin{center}
                   \centering
\begin{longtable}{|c|c|c|c|c|c|c|c|}
  \hline
  $X$&$A_n/P_1 $ &\makecell{$A_n/P_k$\\$(2\leq k\leq \frac{n+1}{2})$}&\makecell{$B_n/P_1$\\$(n=2,3)$}&\makecell{$B_n/P_k$\\$(2\leq k \leq \frac{2n}{3})$}&\makecell{$B_n/P_{n-2}$\\$(n\ge 6)$}&\makecell{$B_n/P_{n-1}$\\$(n\ge 3)$}&$C_n/P_n$\\
\hline
$E$&	$E_{\lambda_n}$&	$E_{\lambda_1}$&	$E_{\lambda_n}$&	$E_{\lambda_1}$&	$E_{\lambda_n}$	&	$E_{\lambda_n}$&	$E_{\lambda_1}$\\

\hline
$X$ &$D_4/P_1$&\makecell{$D_n/P_k$\\$(2\leq k \leq \frac{2n-2}{3})$} &\makecell{$D_n/P_{n-3}$\\$(n\ge 7)$} &\makecell{$D_n/P_{n-2}$\\$(n\ge 4)$} &\makecell{$D_n/P_{n-1}$\\$(\cong D_n/P_n)$}&\makecell{$E_n/P_2$\\$(n=6,7,8)$}&\makecell{$E_n/P_3$\\$(n=6,7,8)$}\\
\hline
$E$ &	$E_{\lambda_4}$&	$E_{\lambda_1}$	&	$E_{\lambda_n}$	&	$E_{\lambda_n}$&	$E_{\lambda_1}$&	$E_{\lambda_1}$&	$E_{\lambda_1}$\\
\hline

$X$ &\makecell{$E_n/P_4$\\$(n=6,7,8)$}& \makecell{$E_n/P_5$\\$(n= 7,8)$} & \makecell{$E_n/P_6$\\ $(n=7,8)$} & $E_8/P_7$ & $F_4/P_1$ & $F_4/P_2$ & \makecell{$G_2/P_1$\\ $(\cong \mathbb{Q}^5)$}\\
\hline
$E$ &	$E_{\lambda_2}$&	$E_{\lambda_n}$	&	$E_{\lambda_n}$	&	$E_{\lambda_8}$ &	$E_{\lambda_4}$ & $E_{\lambda_1}$ & Spinor bundle\\
\hline
$X$ & $G_2/P_2$\\
\cline{1-2}
$E$ & $E_{\lambda_1}$\\
\cline{1-2} 
\caption{}
    \label{Table 2}
\end{longtable}
\end{center}

For the classification of unsplit uniform bundles of minimal ranks, we refer to Table \ref{Table 3}. 

\begin{center}
\begin{longtable}{|c|c|}
        \hline
        $X$ & $E$\\
        \hline
        $\mathbb{P}^n$ & $E_{\lambda_n}(\cong T_{\mathbb{P}^n}(-1))$\\
        \hline
        $A_n/P_k(2\leq k\le \frac{n+1}{2})$ & 
        $E_{\lambda_1}$\\
        \hline $\mathbb{Q}^3,~\mathbb{Q}^5(\cong G_2/P_1),~\mathbb{Q}^6$ & 
        Spinor bundle\\
        \hline
        $B_n/P_k(2\leq k < \frac{2n}{3}),~D_n/P_k(2\leq k <\frac{2n-2}{3})$ & $E_{\lambda_1}$\\
        \hline
        $B_n/P_{n-1}(n>3),~B_n/P_{n-2}(n>6)$ & $E_{\lambda_n}$\\
        \hline
         $B_n/P_n(\cong D_{n+1}/P_{n+1})
$ & $E_{\lambda_1}$\\
        \hline
         $B_3/P_2$ & $E_{\lambda_1},~E_{\lambda_3}$\\
        \hline
        $B_6/P_4$&$E_{\lambda_1},~E_{\lambda_6}$\\
        \hline
        
        $D_n/P_{n-2}(n>4)$ & $E_{\lambda_{n-1}},~E_{\lambda_n}$\\
        \hline
        $D_n/P_{n-3}(n>7)$ & $E_{\lambda_n}$\\
        \hline
        $D_4/P_2$ & $E_{\lambda_1},~E_{\lambda_{3}},~E_{\lambda_4}$\\
        \hline
       $ D_7/P_4$ & $E_{\lambda_1},~E_{\lambda_7}$\\
       \hline $E_n/P_2(n=6,7,8),~E_n/P_3(n=6,7,8)$ & $E_{\lambda_1}$\\
       \hline
       $E_n/P_4(n=6,7,8)$ &
       $E_{\lambda_2}$\\
       \hline
       $E_n/P_5(n=7,8),~E_n/P_6(n=7,8),~E_n/P_7(n=8)$ & $E_{\lambda_n}$\\
       \hline
       $F_4/P_2$ & $E_{\lambda_1}$\\
       \hline
        \caption{}
    \label{Table 3}
    \end{longtable}
\end{center}

We believe the thresholds for uniform bundles to be splitting on most generalised Grassmannians are well known to experts. Here we list these thresholds in a unified table. According to our limited knowledge,  the thresholds for $E_8/P_8$ and $F_4/P_4$ here have not appeared in the existed literature.\\

In \cite{kumar2023nonexistence}, Kumar proposes the following conjecture:\\

\textbf{Conjecture:} Let $X,$ $X'$ be two generalised flag varieties. Assume that $X$ is different from $\mathbb{P}^{2n}$ and $\mathrm{minss~ rank~} X$ is bigger than $\mathrm{maxss~ rank~} X'$, then there are only constant morphisms from $X$ to $X'$ (for the definitions of $\mathrm{minss~ rank~} X$ and $\mathrm{maxss~ rank~} X'$, we refer to \cite{kumar2023nonexistence}).\\

We verify this conjecture for cases under some additional assumptions in Theorem \ref{no map to B,C,D}, where the Dynkin diagram of $X'$  is of type $B,$ $C$ or  $D$ .\\

  Let $p:\mathcal{U}\rightarrow X$ be the universal family of lines on a generalised Grassmannian $X$.  There are three key ingredients in our article: \\
    
        $(1)$ We give a new description of the Chow ring of a generalised flag variety whose Dynkin diagram is of type $B$, $C$ or $D$. This description enables us to classify unsplit uniform bundles on some generalised Grassmannians by using similar methods as in \cite{guyot1985caracterisation}.\\
        
        $(2)$ A given uniform bundle $E$ on $X$ determines a morphism $f_x$ from $p^{-1}(x)$  to some flag variety for any  $x\in X$. Once all $f_x$ are constant morphisms, one can prove $E$ splits. \\
        
         $(3)$   For any vector bundle $V_1,$ $V_2$ on $\mathcal{U}$, if $V_1|_{p^{-1}(x)}$ is simple and is isomorphic to $V_2|_{p^{-1}(x)}$ for any $x\in X$, then $V_1$ is isomorphic to $V_2\otimes p^*L$ for some line bundle $L$ on $X$.  Then we try to classify some quotient bundle $F$ of $p^*(E) $ and show that $E$ is isomorphic to $p_*(F).$\\

    The organization of this article and the highlights of each sections are as follows: \\
    
    In Section $2,$ we give a new description of the Chow rings of some generalised flag varieties and apply it to give a partially affirmative answer to Kumar's conjecture; \\

 In Section $3,$ we prove various criteria for a uniform bundle to be splitting. Especially, we show that if the  rank of an unsplit uniform bundle $E$ on a generalised Grassmannian $X$ is $e.d.(\mathrm{VMRT})+1$, then the splitting type of $E$ is $(\underbrace{1,\dots,1}_{rk(E)-l},\underbrace{0,\dots,0}_{l})$;\\
 
 In Section $4,$  we classify unsplit uniform bundles of minimal ranks. There is a unified proof for the classification when the VMRT of $X$ is a Grassmannian. When the VMRT of $X$ is a product of several factors,  as the Chow rings of some generalised Grassmannians are quite involved, it seems difficult to apply Guyot's methods directly. But   by considering the sub-diagram embeddings, we can reduce the classification to some known cases. The classification for $\mathbb{Q}^5$ needs detailed and robust calculations.\\
    
In this article, we keep the assumptions and keep using the notations as follows.\\

\textbf{Notations:} All varieties are assumed to be smooth,  projective and defined over $\mathbb C.$ For $X$, we denote by $\rho(X)$ the Picard number of $X.$\\

When we consider a rational homogeneous variety $G/P$, we always assume that $G$ is a simple linear  algebraic group and $P$ is a parabolic subgroup of $G$. We use the Bourbaki labeling for the Dynkin diagrams of simple Lie algebras as follows:
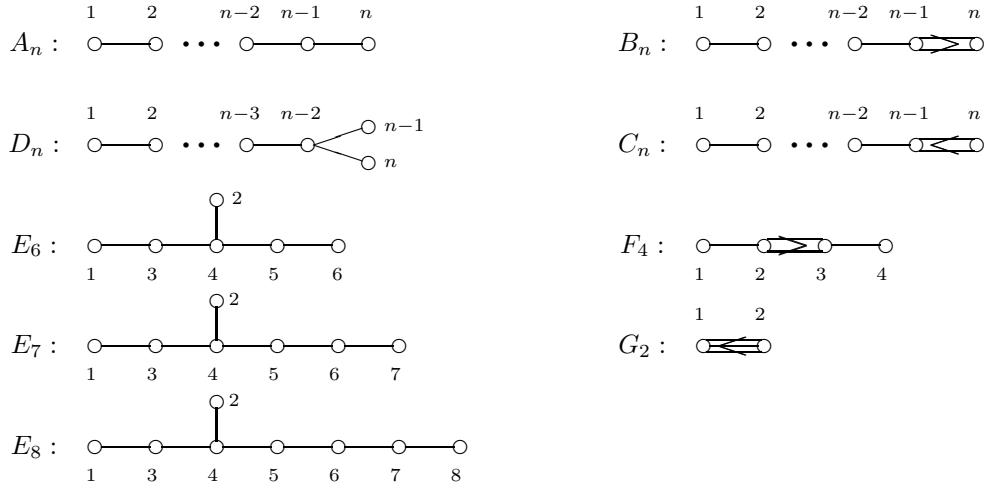
\setlength{\unitlength}{0.4mm}
\begin{center}
\begin{picture}(280,0)(0,120)
\put(10,100){\circle{4}} \put(30,100){\circle{4}}
\put(60,100){\circle{4}} \put(80,100){\circle{4}}
\put(12,100){\line(1,0){16}} \put(40,100){\circle*{2}}
\put(45,100){\circle*{2}} \put(50,100){\circle*{2}}
 \put(62,100){\line(1,0){16}}
 \put(100,100){\circle{4}}
 \put(82,100){\line(1,0){16}}
\put(-10,100){\makebox(0,0)[cc]{$A_n:$}}
\put(7,110){$_1$}
\put(27,110){$_2$}
\put(51,110){$_{n-2}$}
\put(71,110){$_{n-1}$}
\put(97,110){$_{n}$}

\put(210,100){\circle{4}} \put(230,100){\circle{4}}
\put(260,100){\circle{4}} \put(280,100){\circle{4}}
\put(212,100){\line(1,0){16}} \put(240,100){\circle*{2}}
\put(245,100){\circle*{2}} \put(250,100){\circle*{2}}
 \put(262,100){\line(1,0){16}}
 \put(300,100){\circle{4}}
 \put(281,102){\line(1,0){18}}
 \put(281,98){\line(1,0){18}}
 \put(285,103){\line(3,-1){9}}
 \put(285,97){\line(3,1){9}}
\put(190,100){\makebox(0,0)[cc]{$B_n:$}}
\put(207,110){$_1$}
\put(227,110){$_2$}
\put(251,110){$_{n-2}$}
\put(271,110){$_{n-1}$}
\put(297,110){$_{n}$}
 \end{picture}
\end{center}
\vspace{.3cm}

\begin{center}
\begin{picture}(280,20)(0,120)
\put(10,100){\circle{4}} \put(12,100){\line(1,0){16}}
\put(30,100){\circle{4}} \put(40,100){\circle*{2}}
\put(45,100){\circle*{2}} \put(50,100){\circle*{2}}
 \put(60,100){\circle{4}}
\put(62,100){\line(1,0){16}} \put(80,100){\circle{4}}
\put(82,100){\line(3,1){16}} \put(100,94){\circle{4}} \put(82,100){\line(3,-1){16}} \put(100,106){\circle{4}}
\put(-10,100){\makebox(0,0)[cc]{$D_n:$}}
\put(7,110){$_1$}
\put(27,110){$_2$}
\put(51,110){$_{n-3}$}
\put(71,110){$_{n-2}$}
\put(105,106){$_{n-1}$}
\put(105,94){$_{n}$}

\put(210,100){\circle{4}} \put(230,100){\circle{4}}
\put(260,100){\circle{4}} \put(280,100){\circle{4}}
\put(212,100){\line(1,0){16}} \put(240,100){\circle*{2}}
\put(245,100){\circle*{2}} \put(250,100){\circle*{2}}
 \put(262,100){\line(1,0){16}}
 \put(300,100){\circle{4}}
 \put(281,102){\line(1,0){18}}
 \put(281,98){\line(1,0){18}}
 \put(285,100){\line(3,-1){9}}
 \put(285,100){\line(3,1){9}}
\put(190,100){\makebox(0,0)[cc]{$C_n:$}}
\put(207,110){$_1$}
\put(227,110){$_2$}
\put(251,110){$_{n-2}$}
\put(271,110){$_{n-1}$}
\put(297,110){$_{n}$}
 \end{picture}
\end{center}
\vspace{.3cm}

\begin{center}
\begin{picture}(280,20)(0,120)
\put(10,100){\circle{4}} \put(12,100){\line(1,0){16}}
\put(30,100){\circle{4}} \put(32,100){\line(1,0){16}}
\put(50,100){\circle{4}} \put(52,100){\line(1,0){16}}
\put(70,100){\circle{4}} \put(72,100){\line(1,0){16}}
\put(90,100){\circle{4}} \put(50,102){\line(0,1){11}}
\put(50,115){\circle{4}} \put(-10,100){\makebox(0,0)[cc]{$E_6:$}}
\put(7,90){$_1$}
\put(27,90){$_3$}
\put(47,90){$_4$}
\put(67,90){$_5$}
\put(88,90){$_6$}
\put(55,115){$_2$}

\put(210,100){\circle{4}} \put(212,100){\line(1,0){16}}
\put(230,100){\circle{4}} \put(231,102){\line(1,0){18}} \put(231,98){\line(1,0){18}}
\put(250,100){\circle{4}} \put(252,100){\line(1,0){16}}
\put(270,100){\circle{4}}
 \put(190,100){\makebox(0,0)[cc]{$F_4:$}}
 \put(235,103){\line(3,-1){9}}
 \put(235,97){\line(3,1){9}}
\put(207,90){$_1$}
\put(227,90){$_2$}
\put(247,90){$_3$}
\put(267,90){$_4$}
 \end{picture}
\end{center}
\vspace{.3cm}

\begin{center}
\begin{picture}(280,20)(0,120)
\put(10,100){\circle{4}} \put(12,100){\line(1,0){16}}
\put(30,100){\circle{4}} \put(32,100){\line(1,0){16}}
\put(50,100){\circle{4}} \put(52,100){\line(1,0){16}}
\put(70,100){\circle{4}} \put(72,100){\line(1,0){16}}
\put(90,100){\circle{4}} \put(92,100){\line(1,0){16}}
\put(110,100){\circle{4}} \put(50,102){\line(0,1){11}}
\put(50,115){\circle{4}} \put(-10,100){\makebox(0,0)[cc]{$E_7:$}}
\put(7,90){$_1$}
\put(27,90){$_3$}
\put(47,90){$_4$}
\put(67,90){$_5$}
\put(87,90){$_6$}
\put(107,90){$_7$}
\put(54,115){$_2$}

\put(210,100){\circle{4}} \put(230,100){\circle{4}}
\put(211,102){\line(1,0){18}}
\put(212,100){\line(1,0){16}}
\put(211,98){\line(1,0){18}}
\put(190,100){\makebox(0,0)[cc]{$G_2:$}}
\put(207,110){$_1$}
\put(227,110){$_2$}
 \put(215,100){\line(3,-1){9}}
 \put(215,100){\line(3,1){9}}
 \end{picture}
\end{center}
\vspace{.3cm}

\begin{center}
\begin{picture}(280,20)(0,120)
\put(10,100){\circle{4}} \put(12,100){\line(1,0){16}}
\put(30,100){\circle{4}} \put(32,100){\line(1,0){16}}
\put(50,100){\circle{4}} \put(52,100){\line(1,0){16}}
\put(70,100){\circle{4}} \put(72,100){\line(1,0){16}}
\put(90,100){\circle{4}} \put(92,100){\line(1,0){16}}
\put(110,100){\circle{4}} \put(112,100){\line(1,0){16}}
\put(130,100){\circle{4}} \put(50,102){\line(0,1){11}}
\put(50,115){\circle{4}} \put(-10,100){\makebox(0,0)[cc]{$E_8:$}}
\put(7,90){$_1$}
\put(27,90){$_3$}
\put(47,90){$_4$}
\put(67,90){$_5$}
\put(87,90){$_6$}
\put(107,90){$_7$}
\put(127,90){$_8$}
\put(54,115){$_2$}
 \end{picture}
\end{center}
\vspace{2cm}

Let $\mathcal{D}(G)$ be the Dynkin diagram of $G$ and let $\mathrm{D}(G)$ be the set of nodes in the Dynkin diagram $\mathcal{D}(G)$. Let $I$ be a subset of $\mathrm{D}(G)$.
Every parabolic subgroup  of $G$ is uniquely determined by marking a subset of nodes in the Dynkin diagram of $G$. We denote by $P_I$ the parabolic subgroup determined by marking $I$. For example, $P_k$ is the maximal parabolic subgroup corresponding to the $k$-th node in $\mathcal{D}(G).$ When it is necessary, we will make the type of the Dynkin diagram of $G$ explicit. For example, we denote by $B_n/P_I$ the variety $G/P_I$ where $\mathcal{D}(G)$ is of type $B_n$. When $P$ is a maximal parabolic subgroup of $G$, we call the rational variety $G/P$ a generalised Grassmannian.\\

If $X$ is a generalised Grassmannian, we fix a family of rational curves of minimal degree (which is a family of lines) on $X$ and denote it by 
    \[
    \begin{tikzcd}
        \mathcal{U}\arrow[r,"q"]\arrow[d,"p"]& \mathcal{M}\\
        X,
    \end{tikzcd}
    \]
    where $\mathcal{M}$ is the parameter space  and $\mathcal{U}$ is the universal family. For any $x\in X$, the fiber $p^{-1}(x)$ is the VMRT of $X$ which parameterizes the rational curves of minimal degree passing through $x.$\\
    
We say a vector bundle splits when it can be decomposed as a direct sum of line bundles. We say a vector bundle unsplits if it can not be decomposed as a direct sum of line bundles.

\section{Chow rings of some generalised flag varieties and Kumar's conjecture}

 We begin with  analogous descriptions of the cohomology rings of the generalised flag varieties whose Dynkin diagrams are of types $B,C,D$ as in  \cite[Theorem 3.2]{guyot1985caracterisation}. These descriptions are also essentially used in Section $4.$ For positive integers $d_1,\dots,d_s$($d_1<\cdots<d_s\leq n$), we assume 
$$\mathbb{Q}[X_1,\dots,X_{d_1};\dots;X_{d_{s-1}+1},\dots,X_{d_s};]$$ to be the ring of polynomials  symmetrical in $X_{d_{l-1}+1},\dots,X_{d_{l}}$ for each $l$ $(1\leq l\leq s)$. Let $\Sigma_u(X_1,X_2,\dots,X_v)$ be the complete homogeneous symmetric polynomial of degree $u$ in $v$ variables. Let $\sigma_u(X_1,X_2,\dots,X_v)$ be the elementary symmetric polynomial.
\begin{lemma}\label{q-cohomology of Bn/P}
    We have the following isomorphisms:
    $$H^{\bullet}(B_n/P_{d_1,\dots,d_s},\mathbb{Q})\cong H^{\bullet}(C_n/P_{d_1,\dots,d_s},\mathbb{Q})\cong \mathbb{Q}[X_1,\dots,X_{d_1};\dots;X_{d_{s-1}+1},\dots,X_{d_s};]/I,$$
    where $I$ is the ideal generated by the set of polynomials $\Sigma_i(X_1^2,\dots,X_{d_s}^2)$ $(n+1-d_{s}\leq i \leq n)$.
\end{lemma}
\begin{proof}
    Let us firstly recall some facts in the proof of \cite[Theorem 3.2]{guyot1985caracterisation}. In the Chow ring $$A^{\bullet}(A_{n-1}/P_{d_1,\dots,d_s})\cong \frac{\mathbb{Z}[X_1,\dots,X_{d_1};X_{d_1+1},\dots,X_{d_2};\dots;X_{d_s+1},\dots,X_n;]}{(\sigma_i(X_1,\dots,X_n)_{1\leq i \leq n})},$$
    there are relations (see \cite[Page 53, Proof of Theorem 3.2]{guyot1985caracterisation}):
    \begin{align}
        &\sigma_j(X_{d_s+1},\dots,X_n)=(-1)^j\Sigma_j(X_1,\dots,X_{d_s}),~ 1\leq j \leq n-d_s, \label{sigma to Sigma}\\
            & \Sigma_i(X_{d_s+1},\dots,X_n)=0,~n-d_s<i\leq n. \label{Sigma vanish}
    \end{align}
    Therefore the polynomials $\Sigma_i(X_{d_s+1},\dots,X_n)(n-d_s<i\leq n)$ are in the ideal $(\sigma_i(X_1,\dots,X_n)_{1\leq i \leq n})$. There is  an isomorphism
    $$A^{\bullet}(A_{n-1}/P_{d_1,\dots,d_s})\cong \frac{\mathbb{Z}[X_1,\dots,X_{d_1};X_{d_1+1},\dots,X_{d_2};\dots;X_{d_{s-1}+1},\dots,X_{d_s};]}{(\Sigma_i(X_1,\dots,X_{d_s})_{n-d_s+1\leq i \leq n})}.$$
    We are going to use similar relations as (\ref{sigma to Sigma}) and (\ref{Sigma vanish}) in the presentation of the cohomology ring of $B_n/P_{d_1,\dots,d_s}$ and $C_n/P_{d_1,\dots,d_s}$.\\

    For convenience, we now introduce some notations as follows:
    \begin{equation*}
        \begin{aligned}
            &R_1:=\mathbb{Q}[X_1,\dots,X_{d_1};\dots;X_{d_{s-1}+1},\dots,X_{d_s};],\\
            &R_2:=\mathbb{Q}[X_1^2,\dots,X_{d_1}^2;\dots;X_{d_{s-1}+1}^2,\dots,X_{d_s}^2;],\\
            &T_1:=\mathbb{Q}[X_1,\dots,X_{d_1};\dots;X_{d_{s-1}+1},\dots,X_{d_s};X_{d_s+1}^2,\dots ,X_{n}^2;],\\
            &T_2:=\mathbb{Q}[X_1^2,\dots,X_{d_1}^2;\dots;X_{d_{s-1}+1}^2,\dots,X_{d_s}^2;X_{d_s+1}^2,\dots ,X_{n}^2;].
        \end{aligned}
    \end{equation*}
    Let $I_m(m=1,2)$ be the ideal in $R_m$ generated by $\Sigma_i(X_1^2,\dots,X_{d_s}^2)$  $(n+1-d_s\le i\leq n)$. Let $J_k(k=1,2)$ be the ideal in $T_k$ generated by $\sigma_j(X_1^2,\dots,X_n^2)$  $(1\le j\le n)$.
    By \cite[Corollary 3.3, Table 3]{munoz2020splitting}, we have
    $H^{\bullet}(B_n/P_{d_1,\dots,d_s},\mathbb{Q})\cong H^{\bullet}(C_n/P_{d_1,\dots,d_s},\mathbb{Q})\cong T_1/J_1$.  Now we aim to construct an isomorphism $i_1:R_1/I_1\rightarrow T_1/J_1$.\\

    Let $\tau_2: R_2\hookrightarrow T_2$ be the inclusion. Since the polynomials $\Sigma_i(X_1^2,\dots,X_{d_s}^2)$  $(n+1-d_s\le i\leq n)$ are in the ideal $J_2$ as explained in (\ref{Sigma vanish}), the morphism $\tau_2$ maps $I_2$ into $J_2$.
   So according to  \cite[Theorem 3.2]{guyot1985caracterisation}, there is  an isomorphism $i_2:R_2/I_2\cong T_2/J_2$ induced by $\tau_2$. Then $I_1$ is also contained in $J_1$ and we have a morphism $i_1:R_1/I_1\rightarrow T_1/J_1$ induced by the inclusion $\tau_1: R_1\hookrightarrow T_1$, which makes the following diagram 
    commute
    \[
    \begin{tikzcd}
        R_2/I_2\arrow[r,"i_2"]\arrow[d]& T_2/J_2\arrow[d]\\
        R_1/I_1\arrow[r,"i_1"]&T_1/J_1
    \end{tikzcd}
    \]
    Note that $i_2$ is surjective. For any polynomial $f$ that is symmetrical in variables $X_{d_s+1}^2,\dots,X_n^2$, the class $[f]$ in $T_1/J_1$ is contained in the image of $i_1$. So $i_1$ is surjective.\\
    
   
We now construct a section of $i_1$ and hence show that $i_1$ is injective.
   For any $l$, let $f_l$ be an element in $R_2$ such that the class $[f_l] $ is $i_2^{-1}([\sigma_l(X_{d_s+1}^2,\dots,X_n^2)])$.
   Since the elementary symmetric polynomials are algebraically independent, there exists a morphism $\phi:T_2\rightarrow R_2$ satisfying that for any $v(1\leq v\leq s)$ and $u$,
    $\phi(\sigma_u(X_{d_{v-1}+1}^2,\dots,X_{d_v}^2))$ is $\sigma_u(X_{d_{v-1}+1}^2,\dots,X_{d_v}^2)$ and $ \phi(\sigma_l(X_{d_s+1}^2,\dots,X_n^2))$ is $f_l.$ Note that for any element $f\in T_2,$ the class $[\phi(f)]$ is $i_2^{-1}([f])$. If $f$ is in the ideal $J_2$, the class $[\phi(f)]$ is zero, which means that $\phi$ maps $J_2$ into $I_2.$ Especially, the element $\phi(\sigma_j(X_1^2,\dots,X_n^2))$  $(1\le j\le n)$ is a combination of $\Sigma_i(X_1^2,\dots,X_{d_s}^2)$  $(n+1-d_s\le i\leq n)$ with coefficients in $R_2.$  Then there exists a morphism $\Phi: T_1\rightarrow R_1$ satisfying $\Phi\circ \tau_1=id_{R_1}$ and $\Phi|_{T_1}=\phi$. So $\Phi$ maps $J_1$ into $I_1$. Hence $\Phi$ induces a morphism $\varphi: T_1/J_1\rightarrow R_1/I_1$ which is a section of $i_1$.
\end{proof}
Let  $\eta_{d_s}$ be the polynomial $X_{d_s+1}\dots X_n$ for any integer $d_s\le n$.
We assume 
$$\mathbb{Q}[X_1,\dots,X_{d_1};\dots;X_{d_{s-1}+1},\dots,X_{d_s};\eta_{d_s}]$$
 to be the subring of $\mathbb{Q}[X_1,\dots,X_n]$ generated by $\eta_{d_s}$ and  the polynomials symmetrical in $X_{d_{l-1}+1},\dots$
 $,X_{d_{l}}(1\leq l\leq s)$. We now give a  description of the cohomology rings of the generalised Grassmannians whose Dynkin diagrams are of type $D$.
\begin{lemma}\label{q-cohomology of Dn/P}
    For $d_s$$(\ne n-1)$, we have an isomorphism
    $$H^{\bullet}(D_n/P_{d_1,\dots,d_s},\mathbb{Q})\cong \frac{\mathbb{Q}[X_1,\dots,X_{d_1};\dots;X_{d_{s-1}+1},\dots,X_{d_s};\eta_{d_s}]}{(\Sigma_i(X_1^2,\dots,X_{d_s}^2)_{n+1-d_{s}\leq i \leq n}, ~\eta_{d_s}^2-(-1)^{n-d_s}\Sigma_{n-d_s}(X_1^2,\dots,X_{d_s}^2),~X_1X_2\cdots X_n)}.$$
\end{lemma}
\begin{proof}
       We introduce some notations as follows:
    \begin{equation*}
        \begin{aligned}
            &R_1:=\mathbb{Q}[X_1,\dots,X_{d_1};\dots;X_{d_{s-1}+1},\dots,X_{d_s};\eta_{d_s}],\\
             &I_1:=(\Sigma_i(X_1^2,\dots,X_{d_s}^2)_{n+1-d_{s}\leq i \leq n}, ~\eta_{d_s}^2-(-1)^{n-d_s}\Sigma_{n-d_s}(X_1^2,\dots,X_{d_s}^2),~X_1X_2\cdots X_n),\\
            &R_2:=\mathbb{Q}[X_1^2,\dots,X_{d_1}^2;\dots;X_{d_{s-1}+1}^2,\dots,X_{d_s}^2;],\\
            &I_2:=(\Sigma_i(X_1^2,\dots,X_{d_s}^2)_{n+1-d_{s}\leq i \leq n}),\\
            &T_1:=\mathbb{Q}[X_1,\dots,X_{d_1};\dots;X_{d_{s-1}+1},\dots,X_{d_s};X_{d_s+1}^2,\dots ,X_{n}^2;\eta_{d_s}],
            \\&J_1:=(\sigma_i(X_1^2,\dots,X_n^2)_{1\leq i\leq n},~X_1X_2\dots X_n),\\
            &T_2:=\mathbb{Q}[X_1^2,\dots,X_{d_1}^2;\dots;X_{d_{s-1}+1}^2,\dots,X_{d_s}^2;X_{d_s+1}^2,\dots ,X_{n}^2;],\\&J_2:=(\sigma_i(X_1^2,\dots,X_n^2)_{1\leq i\leq n}),
        \end{aligned}
    \end{equation*}
    where $I_m(m=1,2)$ is an ideal in $R_m$ and  $J_k(k=1,2)$ is an ideal in $T_k$. By \cite[Corollary 3.3, Table 3]{munoz2020splitting}, we have
    $H^{\bullet}(D_n/P_{d_1,\dots,d_s},\mathbb{Q})\cong T_1/J_1$.  Now we aim to construct an isomorphism $i_1:R_1/I_1\rightarrow T_1/J_1$.\\

    As in Lemma \ref{q-cohomology of Bn/P},   by \cite[Theorem 3.2]{guyot1985caracterisation}, there is an isomorphism $i_2:R_2/I_2\cong T_2/J_2$ induced by the inclusion $R_2\hookrightarrow T_2$. Note that in  the relation (\ref{sigma to Sigma}), $\sigma_{n-d_s}(X_{d_s+1}^2,\dots,X_n^2)$ is the polynomial $\eta_{d_s}^2$. We have
    $\eta_{d_s}^2-(-1)^{n-d_s}\Sigma_{n-d_s}(X_1^2,\dots,X_{d_s}^2)$ $(=\sigma_{n-d_s}(X_{d_s+1}^2,\dots,X_n^2)-(-1)^{n-d_s}\Sigma_{n-d_s}(X_1^2,\dots,X_{d_s}^2))$ is contained in  $J_2$.
   So $I_1$ is contained in $J_1$ and we have a morphism $i_1:R_1/I_1\rightarrow T_1/J_1$ induced by the inclusion $R_1\hookrightarrow T_1$.  By the same arguments as  in Lemma \ref{q-cohomology of Bn/P},  the morphism $i_1$ is surjective. \\

   We now construct a section of $i_1$. In Lemma  \ref{q-cohomology of Bn/P}, there exists a morphism  $\phi:T_2\rightarrow R_2$ such that $\phi$ maps $\sigma_l(X_{d_s+1}^2,\dots,X_n^2)$ to $(-1)^l\Sigma_l(X_1^2,\dots,X_{d_s}^2)$ for any $l$ ($1\leq l \leq n-d_s $). In the ring  $T_1$, there exists a relation $\eta_{d_s}^2=X_{d_s+1}^2\cdots X_{n}^2$. We denote the quotient map by $\pi:R_1\rightarrow R_1/I_1$ and the inclusion maps by $\delta: R_2\rightarrow R_1$ and $\tau_1:R_1 \rightarrow T_1$. Then in $R_1/I_1$, the class $[\pi \circ\delta\circ\phi(X_{d_s+1}^2\cdots X_{n}^2)]$ is $[\eta_{d_s}^2]$. So there exists a morphism $\Phi: T_1\rightarrow R_1/I_1$ satisfying $\Phi \circ \tau_1 = \pi$ and the following diagram commutes:

    \begin{center}
        \begin{tikzcd}
        T_2 \arrow{d}\arrow{r}{\phi} &R_2\arrow{d}{\pi\circ \delta}\\
        T_1 \arrow{r}{\Phi} & R_1/I_1 
    \end{tikzcd}
    \end{center} Note that $X_1\cdots X_n$ is $X_1\cdots X_{d_s}\cdot\eta_{d_s} $ in $T_1$. The class $[\Phi(X_1\dots X_n)]$ vanishes. So $\Phi$ maps $J_1$ to the zero ideal in $R_1/I_1$ and hence $\Phi$ induces a section of $i_1.$
\end{proof}

In the remaining part of this section, we always assume that  $X$ is a Fano manifold  with $\rho(X)=1$. Let $f$ be a morphism $f:X\longrightarrow Y$. We begin with a simple observation.

\begin{proposition}\label{bigger dimension constant}
   If $\dim (X)$ is greater than $\dim  (Y)$, the morphism $f$ is constant.
\end{proposition}
\begin{proof}
    Since $\rho(X)$ is $1$, if $f$ is non-constant, $f$ will be a finite morphism onto its image.  This is a contradiction to the inequality $\dim(X)>\dim(Y)$.
\end{proof}
When the target variety $Y$ is a (partial) flag variety, for some $X$ of relatively lower dimension, one can also prove that $f$ is constant. By the same arguments as in \cite[Theorem 1]{kumar2023nonexistence} and \cite[Theorem 3.1]{fang2023morphisms}, we have the following proposition.
\begin{proposition}\label{no map fr Q to A/P}
    $(1)$  If $\dim(X)$ is at least $2$ and $Y$ is a generalised complete flag variety $G/B$, the morphism $f$ is a constant map.\\
    
     $(2)$  If $\dim(X)$ is at least  $m$ $(m\ge 2)$ and $Y$ is $A_n/P_I$ where $I$ is $\{1,2,\dots,n-m+2\}$ or $\{m-1,m,\dots,n\}$, the morphism $f$ is a constant map. 
    
\end{proposition}

We now prove a similar result as  Proposition \ref{no map fr Q to A/P} in the following theorem, which gives a partially affirmative answer to Kumar's conjecture in \cite{kumar2023nonexistence}. 

\begin{theorem}\label{no map to B,C,D}
 Let $Y$ be $G/P_I$ where $\mathcal{D}(G)$ is of type $B_n,C_n$ or $D_n$ and $I$ is $\{1,2,\dots,n-m+1\}(m\geq 1)$. If $\mathcal{D}(G)$ is of type $D_n$, we further assume  $m\ne 2$. Then if $\dim(X)$ is at least $2m$, the morphism $f$ is constant.
\end{theorem}
\begin{proof}
   
  Firstly,  in  Lemma \ref{q-cohomology of Bn/P} and Lemma \ref{q-cohomology of Dn/P},    the elements $X_i(1\leq i\leq n-m+1)$ generate the group $H^2(Y, \mathbb{Q})$. The morphism $f$ induces a ring homomorphsim
    $f^{*}:H^{\bullet}(Y,\mathbb{Q})\rightarrow H^{\bullet}(X,\mathbb{Q}).$
    Let $H$ be an  ample generator of $\mathrm{Pic}(X)$. Assume that  $f^*X_i$ is $ a_iH(1\leq i \leq n-m+1)$ where each $a_i$ is a  rational number. Note that the polynomial $\Sigma_m(X_1^2,\dots,X_{n-m+1}^2)$ is  in the defining ideal of $H^{\bullet}(Y,\mathbb{Q})$.  So   we have $\Sigma_m(a_1^2,\dots,a_{n-m+1}^2)H^{2m}=0$. If $\dim(X)$ is at least $2m$, we deduce $\Sigma_m(a_1^2,\dots,a_{n-m+1}^2)=0$. So each $a_i$ vanishes, which implies that $f$ is constant.
\end{proof}

\section{Uniform bundles on a generalised Grassmannian}

For a generalised Grassmannian $G/P$ where $G$ is a simple Lie group, we fix a family of lines on it. A vector bundle $E$ of rank $r$ on $G/P$ is uniform if the restriction of $E$ to every line $L$ is isomorphic to $\oh_L(a_1)\oplus\cdots\oplus\oh_L(a_r)$ with  $a_1\geq \cdots \geq a_r$ and $(a_1,\dots,a_r)$ is  independent of $L$. The $r$-tuple $(a_1,\dots,a_r)$ is called the splitting type of $E$. \\

We prove a slightly stronger result of \cite[Corollary 3.2]{MR4553615}. 
\begin{lemma}\label{tri on G/P}
    Let $X$ be a projective smooth and  rationally connected variety. Assume that a vector bundle $F$ on $X$ is globally generated and its first Chern class $c_1(F)$ vanishes. Then $F$ is trivial. 
    \end{lemma}
\begin{proof}
    For any morphism $\gamma:\mathbb{P}^1\rightarrow X$, the vector bundle $\gamma^* F$ is globally generated and $c_1(\gamma^*F)$ vanishes. \\
    
    Suppose that $\gamma^*F$ splits as $\oplus \oh_{\mathbb{P}^1}(a_i)$. As $\gamma^* F$ is globally generated, each $a_i$ is non-negative. Note that $c_1(\gamma^*F)$ equals to $\Sigma a_i$. Every $a_i$ must be zero, which means that $\gamma^*F$ is trivial. By \cite[Theorem 1.1]{biswas2009vector}, the vector bundle $F$ is trivial. 
\end{proof}

We  consider all generalised Grassmannians and we list their VMRTs in Table \ref{table}. We also associate each variety  $X$ in Table \ref{table} a number $a(X)$. If the VMRT of $X$  is a product of some factors, then   $a(X)$ is the minimal number among the dimensions of these factors. For example,  the VMRT of $Gr(k,n+1)$ is $\mathbb{P}^{k-1}\times \mathbb{P}^{n-k}$. There are two factors $\mathbb{P}^{k-1}$ and $\mathbb{P}^{n-k}$ and $a(Gr(k,n+1))$ is $\min(k-1,n-k)$. If $X$ is $SG(k,2n)$ $(2\le k \le n-1)$ or $F_4/P_3$, the VMRT of $X$ is a fibration over a projective space $\mathbb{P}^{l}$, then $a(X)$ is the minimal number among $l$ and the dimension of the fiber. For the definitions of the effective good divisibility  $e.d.(X)$ and good divisibility $g.d.(X)$ for any $X$, we refer to \cite{munoz2020splitting} and \cite{Pan2015Tri}. For the calculation of $e.d.$$(\mathrm{VMRT})$, we refer to \cite{https://doi.org/10.1002/mana.202300036}, \cite{hu2023effective} and the following Lemma \ref{cal ed for bd}.

\begin{lemma}\label{cal ed for bd}
Let $X$ be a smooth variety with $e.d.(X)=r$. Let $E_1, E_2$ be two vector bundles on $X$ with $\mathrm{rk}(E_1)+\mathrm{rk}(E_2)=n$. Then  $e.d.(\mathbb{P}_X(E_1\oplus E_2))$ is at most $\min(r,n-1)$. In particular, if $g.d.(X)$ is $e.d.(X)$, the number $e.d.(\mathbb{P}_X(E_1\oplus E_2))$ is  $\min(r,n-1).$
\end{lemma}
\begin{proof}

Let us denote the projective bundle by $\pi:M=\mathbb{P}_X(E_1\oplus E_2)\rightarrow X$. Since $e.d.(X)$ is  $r$, there exist effective cycles $x_i(\in A^{i}(X)),x_j(\in A^{j}(X))$ satisfying $x_i\cdot x_j=0$ and $i+j=r+1$. Then $\pi^{*}x_i(\in A^{i}(M))$ and $\pi^*x_j(\in A^{j}(M))$ are effective cycles on $M$ whose intersection  is $0$, which implies that $e.d.(M)$ is at most $r$. \\

The projections $E_1\oplus E_2\rightarrow E_1$ and $E_1\oplus E_2\rightarrow E_2$ induce closed embeddings $s_1:\mathbb{P}_X(E_1)\rightarrow M$ and $s_2:\mathbb{P}_X(E_2)\rightarrow M$. Note that  we have $s_1(\mathbb{P}_X(E_1))\cap s_2(\mathbb{P}_X(E_2))=\emptyset$. Let $n_i$ ($i=1,2$) be the rank of $E_i$. Let $y_i$$(i=1,2)$ be the effective cycle corresponding to $s_i(\mathbb{P}_X(E_i))$. We have $y_1\cdot y_2=0$ with codim $y_1+$codim $y_2=n$. Therefore $e.d.(M)$ is at most $n-1$ and hence $e.d.(M)$ is at most $\min(r,n-1)$. \\

If $g.d.(X)$ is $e.d.(X)$, by \cite[Lemma 5.5]{Pan2015Tri}, the number $g.d.(M)$ is  $\min(r,n-1)$. Since we always have $e.d.(M)\geq g.d.(M)$, the number  $e.d.(M)$ is  $ \min(r,n-1)$. 
\end{proof}

We now calculate $e.d.(\mathrm{VMRT})$ of a generalised Grassmannian $X$.\\

Note that if $X$ is a projective space or a smooth quadric, the number $g.d.(X)$ is  $e.d.(X)$ (\cite[Lemma 5.1, Lemma 5.3, Lemma 5.4]{Pan2015Tri}). For $Gr(k,n+1),SG(k,2n)(2\leq k \leq n-1), OG(k,2n+1)(2\leq k \leq n-1)$ and $OG(k,2n)(2\leq k \leq n-2)$, their VMRTs are projective bundles of some splitting vector bundles over  projective spaces or  smooth quadrics.  So we apply Lemma \ref{cal ed for bd} and obtain the effective good divisibility of their VMRTs as in Table \ref{table}.\\

For $E_7/P_5$ and $E_8/P_5$, we know that $g.d.(Gr(2,5))$ is $3$ by \cite[Lemma 4.4]{munoz2020splitting}. By \cite[Lemma 5.2]{Pan2015Tri}, we have $g.d.(Gr(2,5)\times \mathbb{P}^2)=2$ and $g.d.(Gr(2,5)\times \mathbb{P}^3)=3$.  By similar arguments as above,  there are inequalities $e.d.(Gr(2,5)\times \mathbb{P}^2)\leq 2$ and $~ e.d.(Gr(2,5)\times \mathbb{P}^3)\leq 3$. Then we get $e.d.(Gr(2,5)\times \mathbb{P}^2)= 2$ and $~ e.d.(Gr(2,5)\times \mathbb{P}^3)= 3$. \\

We now show that $g.d.(OG(5,10))$ is  $7$.  The Chow ring of $OG(5,10)$ is 
$$A^{\bullet}(OG(5,10))=\frac{\mathbb{Z}[X_1,X_3]}{(X_3^2-4X_1^3X_3+2X_1^6,12X_1^5X_3-7X_1^8)},$$
where each $X_i$ represents a Schubert subvariety of codimension $i$ (for example, see \cite[Lemma 3.1, Claim a]{du2021vector}). Note that the defining  ideal of $A^{\bullet}(OG(5,10))$ is a homogeneous ideal that is  generated by $X_3^2-4X_1^3X_3+2X_1^6$ and $12X_1^5X_3-7X_1^8$. The polynomial $X_3^2-4X_1^3X_3+2X_1^6$ is irreducible in $\mathbb{Z}[X_1,X_3]$ and the only possible polynomial of degree $7$  in the defining  ideal is $X_1(X_3^2-4X_1^3X_3+2X_1^6)$. So $g.d.(OG(5,10))$ is at least $7$. As $12X_1^5X_3-7X_1^8$ decomposes as $X_1^5(12X_3-7X_1^3)$, the number
$g.d.(OG(5,10))$ is $7$.  Note that $e.d.(OG(5,10))$ is $7$ (see \cite[Theorem 1.1]{hu2023effective}). The effective good divisibility of the VMRT of $E_8/P_6$ is $2$.\\

The VMRT of $F_4/P_4$ is a hyperplane section of $OG(5,10)$ and we denote it by $M$. Let the inclusion map be $i:M\hookrightarrow OG(5,10)$. We firstly show that $e.d.(M)$ is at least $6$. Let $y_1 (\in A^{r}(M))$ and $y_2 (\in A^{s}(M))$ be effective cycles with $y_1\cdot y_2=0$ and $r+s=6$. By the Lefschetz hyperplane section theorem, the restriction map $i^{*}:H^{k}(OG(5,10),\mathbb{Z})\rightarrow H^{k}(M,\mathbb{Z})$ is an isomorphism for $k<9$ and is injective for $k=9$. If codim $y_1$ ($=r$) is less than $3$, $y_1$ is the linearly equivalent class of $\lambda H^r|_M$, where $\lambda$ is positive and $H$ is the ample generator of Pic$(OG(5,10))$. So the equality $y_1\cdot y_2=\lambda H^r|_M\cdot y_2=0$ implies that $y_2$ is zero. If codim $y_1$ is $3$, then there are cycles $x_i$ of codimension $3$ on $OG(5,10)$ such that $i^*(x_i)$ is $y_i$ ($i=1,2$). Assume that in the Chow ring of $OG(5,10)$, $x_i$ is represented by a polynomial $f_i$ of degree $3$. Then the equality $y_1\cdot y_2=i^*(x_1\cdot x_2)=x_1\cdot x_2\cdot M=0$ implies that $X_1 f_1f_2$ is in the defining ideal of $A^{\bullet}(OG(5,10)).$ So $X_1(X_3^2-4X_1^3X_3+2X_1^6)$ divides $X_1 f_1f_2$, which means $\mu (X_3^2-4X_1^3X_3+2X_1^6)$ decomposes as $f_1f_2$ for some rational number $\mu$. But $X_3^2-4X_1^3X_3+2X_1^6$ is indecomposable. So $f_1$ or $f_2$ is zero, hence $e.d.(M)$ is at least $6$.\\

Now we show $e.d.(M)< 7.$ Since $e.d.(OG(5,10))$ is $7$, there exist non-zero effective cycles $x_1$ and $x_2$ with $x_1\cdot x_2=0$ and codim $x_1+$codim $x_2=8$. Suppose that $x_i$ is represented by a polynomial $f_i$ in $A^{\bullet}(OG(5,10))$. Then $f_1f_2$ is a linear combination of $X_1^2(X_3^2-4X_1^3X_3+2X_1^6)$ and $12X_1^5X_3-7X_1^8$. So $X_1^2$ divides $f_1f_2$. Without loss of generality, we may assume that $X_1$ divides $f_1$. Then there exists a cycle $z_1$ such that $z_1\cap M$ equals to $x_1$ in $A^{\bullet}(OG(5,10)).$ Let $y_1$ be $i^*(z_1)$. Let $y_2$ be $i^*(x_2)$. Note that $y_1$ is the intersection $z_1\cap M$, which is effective as $x_1$ is effective. So we have two non-zero effective cycles $y_1$ and $y_2$ in $A^{\bullet}(M)$ such that $y_1\cdot y_2$ is zero and codim $y_1+$codim $y_2$ is $7.$ Then $e.d.(M)$ is $6.$\\

When $X$ is $E_6/P_k,(k=3,4),E_7/P_k(k=3,4,6),E_8/P_k(k=3,4,7)$ or $F_4/P_k(k=2,3)$, the VMRT of $X$ is some smooth fibration over $\mathbb{P}^1$. Then the number $e.d.(\mathrm{VMRT})$ is at most $1$, hence $e.d.(\mathrm{VMRT})$ is $1.$\\

To illustrate our main ideas in this section, we firstly prove a relatively simple result as follows.
\begin{proposition}\label{split of uni}
    Let $X$ be a generalised Grassmannian. Assume that $a(X)$ is bigger than $1$. Then any uniform bundle on $X$ of splitting type  $(a_1,a_2,\dots,a_r)$ with $a_1>a_2>\cdots>a_r$
    splits. 
\end{proposition}

\begin{proof}
    We consider a family of lines on $X$ and keep using the  notations in Section $1$.
    Note that $q:\mathcal{U}\rightarrow \mathcal{M}$ is a $\mathbb{P}^1$-bundle. For any $m\in \mathcal{M}$, the fibre $q^{-1}(m)$ is mapped by $p$ isomorphically to a line $L$ on $X$. Let $E$ be a uniform bundle on $X$ of splitting type  $(a_1,a_2,\dots,a_r)$.
    Then the vector bundle $p^*E|_{q^{-1}(m)}$ splits as $\oh_L(a_1)\oplus \cdots \oplus \oh_L(a_r)$. So for the morphism $q:\mathcal{U}\rightarrow \mathcal{M}$, the vector bundle $p^*E$ has a constant H-N type. By \cite[Corolloary 3.1]{SHN},  $p^*E$ admits a relative Harder-Narasimhan filtration  as follows:
    $$0\subset F_1\subset \cdots \subset F_{r-1}\subset F_{r}=p^*E,$$
    where $F_{i}|_{q^{-1}(m)}$ splits as $\oh_L(a_1)\oplus\cdots\oplus \oh_L(a_i)$. \\
    
    For $x\in X$, the fibre $p^{-1}(x)$ parameterizes all lines on $X$ passing through $x$, which is the VMRT of $X$ at $x$. On each $p^{-1}(x)$, the filtration
    $$0\subset F_1|_{p^{-1}(x)}\subset F_2|_{p^{-1}(x)}\cdots \subset F_{r}|_{p^{-1}(x)}=p^*E|_{p^{-1}(x)}$$
    is a flag bundle, which induces a morphism $  \psi_x:p^{-1}(x)\rightarrow A_{r-1}/B$, where $B$ is the Borel subgroup of $A_{r-1}$. The condition $a(X)>1$ implies that $\mathbb{P}^1$ is not a factor of the VMRT of $X$. If $r$ is $2$, the variety  $A_{r-1}/B$ is $\mathbb{P}^1$ and hence the morphism $\psi_x$ is constant. In the case that $r$ is at least $3$, 
   by \cite[Theorem 1]{kumar2023nonexistence}, the morphism $\psi_x$ is constant. Then all the bundles $F_i|_{p^{-1}(x)}$ are trivial.\\

   Let $F'$ be $p^*E/F_{r-1}$. We have an exact sequence
   $$0\rightarrow F_{r-1}\rightarrow p^*E \rightarrow F'\rightarrow 0.$$
   Let $\pi:A_{r-1}/B\rightarrow Gr(r-1,r)(\cong \mathbb{P}^{r-1})$ be the projection. $F'|_{p^{-1}(x)}$ is the pull back of $\oh_{\mathbb{P}^{r-1}}(1)$ via $\pi \circ \psi_x$, hence is a trivial line bundle for every $x\in X$. By the base change theorem, $p_*F_{r-1}$ is a vector bundle of rank $r-1$ and $p_*F'$ is a line bundle on $X$. \\
   
   There is a long exact sequence
   $$0\rightarrow p_*F_{r-1}\rightarrow p_{*}p^*E \rightarrow p_*F'\rightarrow R^{1}p_*F_{r-1}\rightarrow \cdots .$$
   We note the equation $(R^1p_*F_{r-1})_x=H^{1}(p^{-1}(x),F_{r-1}|_{p^{-1}(x)})$. As $F_{r-1}|_{p^{-1}(x)}$ is a trivial bundle and $p^{-1}(x)$ is rationally connected, the cohomology  $H^{1}(p^{-1}(x),F_{r-1}|_{p^{-1}(x)})$ is $0$ and hence $R^1p_*F_{r-1}$ vanishes. So the sequence
  $ 0\rightarrow p_*F_{r-1}\rightarrow E \rightarrow p_*F' \rightarrow 0$ is exact.\\
  
   For any line bundle $\mathcal{L}$ on $X$, the cohomology group $H^{i}(X,\mathcal{L})$ vanishes for any integer $i$ satisfying $0<i<\dim(X)$. If $r$ is $2$, the cohomology $H^1(X,(p_*F')^{\vee}\otimes p_*F_{r-1})$ vanishes, which implies that $E$ is the direct sum of line bundles $p_*F_{r-1}\oplus p_*F'$. Now assume that $r$ is at least $3$. As $p_*F_{r-1}$ is a uniform bundle of splitting  type $(a_1,\dots,a_{r-1})$, by induction,  $p_*{F_{r-1}}$ splits and hence $H^1(X,(p_*F')^{\vee}\otimes p_*F_{r-1})$  vanishes. So $E$  splits. 
   \end{proof}

We can strengthen Proposition \ref{split of uni} and   \cite[Theorem 4.5]{fang2023morphisms} as follows.

\begin{proposition}\label{split of uni bd}
    Assume $X$ is a generalised Grassmannian. Let $E$ be a uniform bundle on $X$ of splitting  type 
    $$(\underbrace{a_1,\dots,a_1}_{j},a_2,\dots,a_{r-j+1}),a_1>a_2>\cdots>a_{r-j+1}~(r\geq j+1,j>0).$$
    If $j$ is smaller than $a(X)$,  $E$ splits.
\end{proposition}
\begin{proof}
    We follow the  arguments  in the proof of  Proposition \ref{split of uni} and use the notations there. The relative Harder-Narasimhan filtration of $p^*E$ is now of the form
    $$0\subset F_1\subset \cdots \subset F_{r-j}\subset F_{r-j+1}=p^*E,$$
    where $F_{i}|_{q^{-1}(m)}$ splits as $\oh_L(a_1)^{j}\oplus \oh_L(a_2)\cdots\oplus \oh_L(a_i)$. On each $p^{-1}(x)$, the filtration
    $$0\subset F_1|_{p^{-1}(x)}\subset F_2|_{p^{-1}(x)}\cdots \subset F_{r-j+1}|_{p^{-1}(x)}=p^*E|_{p^{-1}(x)}$$ induces a morphism 
    $$\psi_x:p^{-1}(x)\rightarrow A_{r-1}/P_{j,j+1,\dots,r-1}.$$
    Note that the condition $a(X)\geq j+1$ ensures that $p^{-1}(x)$ is covered by Fano manifolds of Picard number one whose dimensions are greater than $j$, then $\psi_x$ is constant by Theorem \ref{no map fr Q to A/P}. \\
    
    Let $F'$ be $p^*E/F_{r-j}$. If $r$ is $j+1$, $p_*F_{r-j}$ is a uniform bundle of splitting  type $(a_1,a_1,\dots,a_1)$, then $p_*F_{r-j}\otimes \oh_X(-a_1)$ is trivial by \cite[Theorem 3.4]{Pan2015Tri}. Since $p_*F'$ is a line bundle, we have $H^{1}(X,(p_*F')^{\vee}\otimes p_*F_{r-j})=0$ and $E$ splits. Now assume that $r$ is at least $j+2$, as $p_*F_{r-j}$ is a uniform bundle of the type $(a_1,\dots ,a_1,a_2,\dots,a_{r-j})$, by induction, $p_*{F_{r-j}}$ splits. So $E$ splits. 
\end{proof}
By observing that the effective good divisiblility of $A_r/P_{I}$ is $r$, which is independent of the choice of $I$, we prove the following theorem.

\begin{theorem}\label{ed to split}
    Let $X$ be a generalised Grassmannian $G/P$ and let $E$ be a uniform bundle of rank $r$ on $X$. If  $r$ is at most $ e.d.(\mathrm{VMRT})$,  $E$ splits. 
\end{theorem}
\begin{proof}

   We may assume that $r$ is at least $2$ and 
    $$E|_L\cong \oh_L(a_1)^{i_1}\oplus\cdots \oplus \oh_L(a_m)^{i_m-i_{m-1}}\oplus \oh(a_{m+1})^{r-i_m},~a_1>\cdots>a_m>a_{m+1}$$
    for each line $L\subset X$, where we set $m\geq 1$ and $i_0=0$. As in the proof of Proposition \ref{split of uni}, the relative Harder-Narasimhan filtration induces a morphism
    $$\psi_x:p^{-1}(x)\rightarrow A_{r-1}/P_{i_1,i_2,\dots,i_m}.$$
    Now $\psi_x$ is constant by \cite[Theorem 1.3]{https://doi.org/10.1002/mana.202300036}, then both $p_* F_{m}$ and $p_*F'$ are uniform bundles (here we use the same notations as in the proof of Proposition \ref{split of uni}). By induction, they split as direct sums of line bundles. Since  $H^1(X,(p_*F')^{\vee}\otimes (p_*F_{m}))$ vanishes, the vector bundle $E$  splits.    
\end{proof}
For some generalised Grassmannians, we show that the upper bounds $e.d.(\mathrm{VMRT})$ are optimal.

\begin{theorem}\label{optimal}
    Let $X$ be a generalised Grassmannian in Table \ref{Table 2}. There exists an
  unsplit uniform bundle $E_{\lambda}$ on $X$ of rank $e.d.(\mathrm{VMRT})+1$, where  $E_{\lambda}$ is the
    irreducible homogeneous bundle corresponding to some highest weight $\lambda$.\\
\end{theorem}

\begin{proof}
    We note that for any generalised Grassmannian $X=G/P_j$ in  Table \ref{Table 2}, the marking root $\alpha_j$ is a long root. So there is only one $G$-orbit in the family of lines on $X$ by \cite[Theorem 4.3]{landsberg2003projective}. This implies any homogeneous bundle on $X$ is a uniform bundle. As an irreducible homogeneous bundle is stable (see \cite[Corollary 12.9]{ottaviani1995rational}), every irreducible homogeneous bundle $E$ on $X$ is unsplit. Then by Weyl's formula, one can show that $\mathrm{rk}(E_{\lambda})$ is $e.d.(\mathrm{VMRT})+1$ for any $E_{\lambda}$ in  Table \ref{Table 2}.
\end{proof}

On a generalised Grassmannian $X$, if $E$ is an unsplit uniform bundle on $X$ of rank $e.d.(\mathrm{VMRT})+1$,  there are some restrictions on the splitting type of $E$. During the discussions with Occhetta, he teaches us the statement and the main ideas of the following proposition.
\begin{proposition}\label{only zeros and ones in the type}
    Let $X$ be a generalised Grassmannian. Let $E$ be an unsplit uniform bundle on $X$ whose rank $r$ is $e.d.(\mathrm{VMRT})+1$. Assume that the splitting type of $E$  is $(a_1,\dots,a_k,0,\dots,0)(a_1\geq \cdots \geq a_k >0)$, then $a_i$ is $1$ for any $i$ $(1\leq i \leq k)$.
\end{proposition}
\begin{proof}
    Let $a$ $ (=r-k)$ be the number of zeros in the splitting type of $E$. We consider the universal family of lines on $X$. For any $x\in X$ and any line $C$ passing through $x$, we consider the minimal sections of $\mathbb{P}_C(E|_C)$. These minimal sections covers a linear subspace $\mathbb{P}((\mathcal{O}_C^{\oplus a})_x)$ of $\mathbb{P}(E_x)$. So there exists a morphism   $\rho_x: p^{-1}(x)\rightarrow \mathbb{G}(a-1,\mathbb{P}(E_x))$ which maps  $[C]$  to  $\mathbb{P}((\mathcal{O}_C^{\oplus a})_x).$ \\

 We denote the relative H-N factor corresponding to $(0,\dots,0)$ of $p^*(E)$ by $F.$ Then there exists a short
    exact sequence 
    \begin{align}\label{exact in zero one lemma}
        0\rightarrow E_1 \rightarrow p^*E \rightarrow F \rightarrow 0,
    \end{align}
    where $F$ is a bundle of rank $a$. For any $x\in X$,  the exact sequence 
    $0\rightarrow E_1|_{p^{-1}(x)} \rightarrow E_x\otimes \oh_{p^{-1}(x)} \rightarrow F|_{p^{-1}(x)} \rightarrow 0 $ 
     induces a morphism $\psi_x: p^{-1}(x)\rightarrow \mathbb G(a-1, \mathbb{P}(E_x))$ which maps any point $z\in p^{-1}(x)$ to $\mathbb{P}(F_z)$. Let $q(z)$ be $y.$ Then $p$ maps $q^{-1}(y)$ isomorphically to the line $C$ parametrized by $y.$ So $F_z$ is the vector space $(F|_{q^{-1}(y)})_z$ which is exactly the vector space $(\mathcal{O}_C^{\oplus a})_x $. Then the morphisms $\psi_x$ and $\rho_x$ are the same.\\

    Let $\overline{\mathcal{M}}$ be the family of minimal sections over lines in X. We claim that the family $\overline{\mathcal{M}}$ is unsplit and dominating.\\
    
    If the claim is true, then there are two methods to calculate the dimension of $\overline{\mathcal{M}}$. Since $\overline{\mathcal{M}}$  parameterizes minimal sections over lines on $X$, $\dim \overline{\mathcal{M}}$ should be $\dim \mathcal{M}+a-1.$ As $\overline{\mathcal{M}}$ is unsplit and dominating, $\dim \overline{\mathcal{M}}$ should be $ \dim \mathbb{P}(E)-K_{\mathbb{P}(E)}\cdot \overline{C}-3$ where $\overline{C}$ is a minimal section over a line in $X$. So
    we have an equation:
    \begin{align}\label{dimension formula}
        \dim \mathbb{P}(E)-K_{\mathbb{P}(E)}\cdot \overline{C}-3=\dim \mathcal{M}+a-1.
    \end{align}
    Note that $\dim \mathcal{M}$ is $\dim X-K_X\cdot C-3$ ($C$ is a line in $X$). We deduce that 
     $(a_1+\cdots+a_k)-k$ is $0$, hence every $a_i$ must be $1$.\\

    Now we prove the claim. By \cite[Proposition 5.4]{munoz2021uniform}, the family $\overline{\mathcal{M}}$ is unsplit. If $\overline{\mathcal{M}}$ is not dominating, then for a general $x\in X$, the linear subspaces parameterized by the image of $\rho_x$ do not cover $\mathbb{P}(E_x)$. Let $L$ be the locus of these linear subspaces. Let $z\in \mathbb{P}(E_x)$ be a point not contained in $L$. Let $H\subset \mathbb{P}(E_x)$ be a hyperplane that does not contain $z$. The linear projection $\pi_z:\mathbb{P}(E_x)\dashrightarrow H$ from $z$ to the hyperplane $H$ maps every linear subspace of dimension $a-1$ that does not contain $z$ to a linear subspace in $H$ of dimension $a-1.$ Then there is a rational map $\pi:\mathbb{G}(a-1,\mathbb{P}(E_x))\dashrightarrow \mathbb{G}(a-1,H) $  and the composition $\pi \circ \rho_x$ is a morphism $p^{-1}(x)\rightarrow \mathbb{G}(a-1,H)$. Since $e.d.(\mathbb{G}(a-1,H))(=r-2)$ is smaller than $e.d.(\mathrm{VMRT})(=r-1)$, the morphism $\pi \circ \rho_x$ is constant by \cite[Theorem 1.3]{https://doi.org/10.1002/mana.202300036}. Note that the fiber of $\pi$ is $\mathbb{P}^a\backslash \mathbb{P}^{a-1}$, which is an affine variety. We deduce that $\rho_x$ is constant for general $x\in X$.
Therefore $\psi_x$ is constant for general $x\in X$.  By the universal property for Grassmannians, we have $E_1|_{p^{-1}(x)}$ and $F|_{p^{-1}(x)}$ are trivial bundles. In particular, both $c_1(F|_{p^{-1}(x)})$ and $c_1(E_1|_{p^{-1}(x)})$ vanish. Since the Chern classes of $F|_{p^{-1}(x')}$ do not vary  for $x'\in X$, the first Chern class $c_1(F|_{p^{-1}(x')})$ vanishes for each $x'\in X$. By Lemma \ref{tri on G/P}, $F|_{p^{-1}(x')}$ is trivial. By applying similar argument to $E_1^{\vee}$, $E_1|_{p^{-1}(x')}$ is also trivial for each $x'\in X$. Then as in Proposition \ref{split of uni},  we have an exact sequence $0\rightarrow p_*E_1 \rightarrow E \rightarrow p_*F \rightarrow 0$. Both $p_*E_1$ and $p_*F$ are uniform bundles on $X$ of lower ranks, so they split by Theorem \ref{ed to split}. Then $E$ splits, which is a contradiction.
\end{proof}

In \cite[Theorem 1.1]{occhetta2023morphisms}, Occhetta and Tondelli prove that every morphism  $\varphi:Gr(l,m)\rightarrow Gr(k,m)$ is constant under the assumptions $l\ne 1,m-1 $  and $l\ne k, m-k$. It has an application to uniform bundles on the varieties whose VMRTs are Grassmannians. 
\begin{proposition}\label{split of uni on VMRT Grass}
    Let $X$ be a generalised Grassmannian $G/P$ whose VMRT is $Gr(l,r+1)$ with $l\ne 1,r$. Let $E$ be an unsplit uniform bundle on $X$ of rank $r+1$. Then by twisting $E$ or taking the dual sheaf of $E$ if necessary, the splitting type of $E$ is  $(\underbrace{1,\dots,1}_{r+1-l},\underbrace{0,\dots,0}_{l}).
    $
\end{proposition}
\begin{proof}
    By Proposition \ref{only zeros and ones in the type}, by twisting $E$ suitably, we may assume the splitting type of $E$ is  $(\underbrace{1,\dots,1}_{r+1-k},\underbrace{0,\dots,0}_{k}).
    $ . The relative Harder-Narasimhan filtration of $p^*(E)$ induces a morphism 
    $$\psi_x:p^{-1}(x)\rightarrow Gr(k,r+1).$$
    If $k$ is not $l$ or $r+1-l$, then $\psi_x$ is constant by \cite[Theorem 1.1]{occhetta2023morphisms}. Then by similar arguments as in Proposition \ref{split of uni}, $E$ splits. So $k$ is $l$ or $k$ is $r+1-l$. By taking the dual sheaf if necessary, we may assume that $k$ is $l$.
\end{proof}

\section{Classification of unsplit uniform bundles of  minimal ranks}

In this section, we will prove  Theorem \ref{main1}, $(3)$. Let us make some conventions first. We always classify unsplit uniform bundles up to twisting by some line bundles. On any generalised Grassmannian $X$, we denote by $E_{\lambda}$ the irreducible homogeneous bundle which corresponds to the highest weight $\lambda$. When we consider the family of lines on $X$, we will keep using the notations in Section $1.$\\

The following lemma will be frequently used in the classification.
\begin{lemma}\label{fiber to global}
    Let $g:Y\rightarrow Z$ be a smooth projective morphism between smooth projective varieties. Assume that  $E_1$ and $E_2$ are two vector bundles on $Y$ such that for each $z\in Z$, $E_1|_{g^{-1}(z)}$ is isomorphic to $ E_2|_{g^{-1}(z)}$ and  the bundle $E_1|_{g^{-1}(z)}$ is simple. Then there is an isomorphism 
    $E_1\cong E_2\otimes g^*L$
    for some line bundle $L$ on $Z$. 
\end{lemma}
\begin{proof}
    We denote the sheaf $g_*(E_1\otimes E_2^{\vee})$ by $L$. As the bundle $E_1|_{g^{-1}(z)}(\cong E_2|_{g^{-1}(z)})$ is simple for each $z\in Z$, we have $$H^0(g^{-1}(z),(E_1\otimes E_2^{\vee})|_{g^{-1}(z)})=\mathbb{C}$$
    and hence $L$ is a line bundle by the base-change theorem. We consider the composition of morphisms
    $$g^*g_*(E_1\otimes E_2^{\vee})\otimes E_2\rightarrow E_1\otimes E_2^{\vee}\otimes E_2 \rightarrow E_1,$$where the first map is induced by the natural transformation $g^*g_*\rightarrow id$ and the second map is induced by the trace map $E_2^{\vee}\otimes E_2\rightarrow \oh_{Y}$. Then we have a morphism $g^*L\otimes E_2\rightarrow E_1$ which is an isomorphism when restricting to each fiber $g^{-1}(z)$. So  we get an isomorphism $g^*L\otimes E_2\cong E_1$.
\end{proof}
\subsection{The cases that the VMRT of  $X$ is a Grassmannian}
\begin{theorem}\label{Sharp bundle on X VMRT Grassmannnian}
    Let $X$ be a generalised Grassmannian whose VMRT is $Gr(l,r+1)$ with $l\ne 1, r$. Let $E$ be an unsplit uniform bundle on $X$ of rank $r+1$.
    \begin{itemize}
        \item If $X$ is $ \mathbb{Q}^6$,  $E$ is one of the two spinor bundles $S',~S'' (~S'^{\vee}\cong S''(1))$.
        \item If $X$ is $OG(n,2n),~E_6/P_2,~E_7/P_2$ or $E_8/P_2$,  $E$ is the bundle $E_{\lambda_1}$ or its dual $E_{\lambda_1}^{\vee}$.
    \end{itemize}
\end{theorem}
\begin{proof}
   Without loss of generality,  we may assume $l\leq \frac{r+1}{2}$. For the family of lines on $X$,
    every $p$-fiber $p^{-1}(x)$ is isomorphic to $Gr(l,r+1)$. By Proposition \ref{split of uni on VMRT Grass}, we may assume the splitting type of $E$ is
    $(\underbrace{1,\dots,1}_{r+1-l},\underbrace{0,\dots,0}_{l}).
    $
    
    We firstly  consider the  case: $l<\frac{r+1}{2}$. We denote the relative H-N factor corresponding to $(0,\dots,0)$ of $p^*E$ by $F$. Then there is a short exact sequence 
    \begin{align}\label{exact}
      0\rightarrow E_1 \rightarrow p^*E \rightarrow F \rightarrow 0 ,  
    \end{align}
    where $E_1$ is a vector bundle of rank $r+1-l$ and $F$ is a bundle of rank $l$.\\

   For any $x\in X$, by restricting the above exact sequence (\ref{exact}) to $p^{-1}(x)$, we get the exact sequence 
    $$0\rightarrow E_1|_{p^{-1}(x)} \rightarrow \oh_{p^{-1}(x)}^{\oplus r+1} \rightarrow F|_{p^{-1}(x)} \rightarrow 0 ,$$ 
    which induces a morphism $\psi_x: p^{-1}(x)\rightarrow Gr(r+1-l,r+1)$. Let $U$ and $Q$ be the universal subbundle and quotient bundle on $Gr(r+1-l,r+1)$. By the universal property for Grassmannians, we have $E_1|_{p^{-1}(x)}\cong \psi_x^*U, ~~F|_{p^{-1}(x)}\cong \psi_x^*Q.$
    Since $p^{-1}(x)$ is $ Gr(r+1-l,r+1)$,  $\psi_x$ is either a constant morphism or an isomorphism by \cite[Theorem 1.1]{occhetta2023morphisms}. If $\psi_{x_0}$ is constant for some $x_0\in X$, then $F|_{p^{-1}(x_0)}(\cong \psi_{x_0}^*Q)$ is trivial. Since the Chern classes of $F|_{p^{-1}(x)}$ do not vary  for $x\in X$, the first Chern class $c_1(F|_{p^{-1}(x)})$ vanishes for each $x\in X$. By Lemma \ref{tri on G/P}, $F|_{p^{-1}(x)}$ is trivial and $\psi_x$ is constant for all $x\in X$. The bundle $E$ splits by the similar arguments as in Proposition \ref{split of uni}. \\

     Therefore $\psi_x$ is an isomorphism for each $x\in X$. By the assumption $l<\frac{r+1}{2}$, $\psi_x$ is given by a linear action of $\mathrm{SL}(r+1)$(for example, see \cite[Theorem 1,1]{cowen1989automorphisms}). 
     Since both $U$ and $Q$ are homogeneous bundles, we have isomorphisms
    $E_1|_{p^{-1}(x)}\cong\psi_x^*U\cong U~\text{and} ~ F|_{p^{-1}(x)}\cong \psi_x^*Q\cong Q$. 
    Since  $H^0(Gr(l,r+1),U)$ and $H^1(Gr(l,r+1),U)$ vanish, the sheaves 
    $p_*E_1$ and $R^1p_*E_1$ vanish, which implies an isomorphism $E\cong p_*F$. \\

   Let $E'$ be another  unsplit uniform bundle of rank $r+1$. We may assume the last relative H-N factor of $p^*(E)$ is $F'$ whose rank is $l$. Then  we have $F|_{p^{-1}(x)}\cong  F'|_{p^{-1}(x)}\cong Q$ for each $x\in X$. As the universal quotient bundle $Q$ is simple, by Lemma \ref{fiber to global}, $F$ is isomorphic to $F'\otimes p^*L$ for some line bundle $L$ on $X$. Then there are isomorphisms $E\cong p_*F \cong p_*F'\otimes L \cong E'\otimes L$. So up to twisting by line bundles, $E$ is unique.\\

   We now consider the case: $l=\frac{r+1}{2}$.
   We use the same notations as above. By Proposition \ref{split of uni on VMRT Grass},  both $E_1$ and $F$ are bundles of rank $l$. By similar arguments as above, the induced morphism $\psi_x$ is an isomorphism for each $x\in X$. The isomorphism $\psi_x$ is either given by a linear action of $SL(2l)$ or the composition of a linear action and the dual map $\Lambda: Gr(l,2l)\rightarrow Gr(l,2l)$ (see \cite[Section 2]{cowen1989automorphisms}). Note the isomorphisms $\Lambda^*U\cong Q^{\vee}$ and $\Lambda^*Q\cong U^{\vee}$. As the Chern classes of $F|_{p^{-1}(x)}$ (resp. $E_1|_{p^{-1}(x)}$) do not vary for  $x\in X$ and the Chern polynomials of $Q$ and $U^{\vee}$ are different for $l>1$ (see \cite[Proposition 5.4]{guyot1985caracterisation}), we have the following two possibilities
    $$E_1|_{p^{-1}(x)}\cong U,~F|_{p^{-1}(x)}\cong Q\text{ for all }x\in X~\text{  or  }~E_1|_{p^{-1}(x)}\cong Q^{\vee},~F|_{p^{-1}(x)}\cong U^{\vee}\text{ for all }x \in X.$$
    By dualizing $E$ if necessary, we can further assume that the first possibility holds. Then the same arguments as in the proof of the case $l<\frac{r+1}{2}$ apply.
\end{proof}

\subsection{The case that $X$ is $\mathbb{Q}^5$}
\begin{theorem}\label{sharp bundle on Q5}
    Let $X$ be  $\mathbb{Q}^5$ and let $E$ be an unsplit uniform vector bundle on $X$ of rank $4$. Then $E$ is the spinor bundle on $\mathbb{Q}^5$.
\end{theorem}
\begin{proof}
    Due to Proposition \ref{only zeros and ones in the type},  we may assume that the splitting type of $E$ is one of the following types
    $(0,0,0,0),~(0,0,0,-1), ~(0,0,-1,-1)$ (take the dual sheaf of $E$ if necessary). We can rule out the possibility $(0,0,0,0)$ by \cite[Theorem 3.4]{Pan2015Tri}.\\
    
    The universal family $\mathcal{U}$ of lines on $X$ is $B_3/P_{1,2}$, the parameter space $\mathcal{M}$ is $B_3/P_2$ and the fiber of $p$ is $\mathbb{Q}^3$. By Lemma \ref{q-cohomology of Bn/P}, we list the cohomology rings as follows
    \begin{equation*}
        \begin{aligned}
            &H^\bullet(X,\mathbb{Q})=\mathbb{Q}[X_1]/(X_1^6),~H^{\bullet}(\mathcal{U},\mathbb{Q})=\mathbb{Q}[X_1,X_2]/(\Sigma_d(X_1^2,X_2^2)_{2\leq d\leq 3}),\\
            &H^{\bullet}(\mathcal{M},\mathbb{Q})=\mathbb{Q}[X_1+X_2,X_1X_2]/(\Sigma_d(X_1^2,X_2^2)_{2\leq d\leq 3}).
        \end{aligned}
    \end{equation*}
    Note that the degrees of  the generators in the defining ideals are at least $4$. We denote the Chern classes $c_t(p^*E)(1\leq t \leq 4)$ of $p^*E$ by $\mu_tX_1^t$ for some rational numbers $\mu_t$.\\

    \textbf{Case I}: the splitting type of $E$ is $(0,0,0,-1)$.\\

    The relative H-N filtration induces an exact sequence
    $$0\rightarrow q^*G_1 \rightarrow p^*E \rightarrow q^*G_2\otimes p^*\oh_X(-1) \rightarrow 0,$$ where the rank of $G_1$ is $3$  and $G_2$ is a line bundle. Let $F_1$ be $q^*G_1$ and let $F_2$ be $q^*G_2\otimes p^*\oh_X(-1)$. Now we denote the Chern classes of $F_1$ and $F_2$ as in the following equations: 
    \begin{equation*}
        \begin{aligned}
            &c_1(F_1)=a(X_1+X_2),~c_1(F_2)=b(X_1+X_2)-X_1;\\
            &c_2(F_1)=a_2X_1X_2+a_2'(X_1^2+X_2^2);\\
            &c_3(F_1)=a_3(X_1+X_2)X_1X_2+a_3'(X_1+X_2)(X_1^2+X_2^2).
        \end{aligned}
    \end{equation*}
    The equality $c_1(p^*E)=c_1(F_1)+c_1(F_2)=(a+b-1)X_1+(a+b)X_2=\mu_1X_1$ implies $a=-b$. By comparing the coefficients of $X_2^2$ and $X_1X_2$ on both sides of the equation $\mu_2X_1^2=c_2(p^*E)=c_2(F_1)+c_1(F_1)c_1(F_2)=a_2X_1X_2+a_2'(X_1^2+X_2^2)+[-b(X_1+X_2)][(b-1)X_1+bX_2]=(-b^2+a_2'+b)X_1^2
            +(-2b^2+a_2+b)X_1X_2+(-b^2+a_2')X_2^2$, 
    we deduce $a_2'=b^2$ and $~a_2=2b^2-b$. By comparing the coefficients of the equation $\mu^3X_1^3=c_3(p^*E)=c_3(F_1)+c_2(F_1)c_1(F_2)=[a_3'+b^2(b-1)]X_1^3+[(2b^2-b)(b-1)+b^3+a_3+a_3']X_1^2X_2
            +[b^2(3b-2)+a_3+a_3']X_1X_2^2+[b^3+a_3']X_2^3$, 
    we have $a_3'=-b^3$, $a_3=(2b^2-b)(1-b)$ and $(2b^2-b)(b-1)+b^3=b^2(3b-2)$, which implies that $b$ is $0$ or $b$ is $1$. \\
    
    If $b$ is $0$, then we have $a=a_2=a_2'=a_3=a_3'=0$, which means that the $c_1(F_1|_{p^{-1}(x)})$ and $c_1(F_2|_{p^{-1}(x)})$ vanish for $x\in X$. By Lemma \ref{tri on G/P}, $(F_1|_{p^{-1}(x)})^{\vee}$ and $F_2|_{p^{-1}(x)}$ are trivial. Similar arguments in Proposition \ref{split of uni} show that $E$ splits.\\

    If $b$ is $1$, then we have $a_3'=-1$ and $a_3=0$. So by the equation $c_4(p^*E)=c_3(F_1)c_1(F_2)$, we have
    $\mu_4X_1^4=-(X_1+X_2)(X_1^2+X_2^2)X_2  ~(\mathrm{mod} ~X_1^4+X_1^2X_2^2+X_2^4).$
    Note the coefficient of $X_1^3X_2$ in $(X_1+X_2)(X_1^2+X_2^2)X_2$ (hence on the right hand of the above equation) is not $0$. We exclude the possibility that the splitting type of $E$ is $(0,0,0,-1)$.\\

\textbf{Case II}: the splitting type of $E$ is $(0,0,-1,-1)$.\\

Let $0\rightarrow E_1 \rightarrow p^*E \rightarrow E_2 \rightarrow 0$ be the exact sequence induced by the relative H-N filtration. There are bundles $G_1$ and $G_2$ of rank $2$ on $\mathcal{M}$ satisfying $E_1=q^*G_1$ and $E_2=q^*G_2\otimes p^*\oh_X(-1)$. We denote the Chern classes in the following equations 
    \begin{equation*}
        \begin{aligned}
            &c_1(E_1)=a(X_1+X_2),~c_2(E_1)=a_2X_1X_2+a_2'(X_1^2+X_2^2),\\
            &c_1(E_2)=b(X_1+X_2)-2X_1,~c_2(E_2)=b_2X_1X_2+b_2'(X_1^2+X_2^2)-b(X_1+X_2)X_1+X_1^2.
        \end{aligned}
    \end{equation*}
    From $c_1(p^*E)=c_1(E_1)+c_1(E_2)=(a+b-2)X_1+(a+b)X_2=\mu_1X_1$, we have $b=-a$. From $c_2(p^*E)=c_2(E_1)+c_1(E_1)c_1(E_2)+c_2(E_2)$, we have
   \begin{align}
            \mu_2X_1^2=&a_2X_1X_2+a_2'(X_1^2+X_2^2)+[-a(X_1+X_2)-2X_1][a(X_1+X_2)]\nonumber\\
            +&b_2X_1X_2+b_2'(X_1^2+X_2^2)+a(X_1+X_2)X_1+X_1^2\label{formula 1 in iijj}.
   \end{align}
    The coefficient of $X_1X_2$ in (\ref{formula 1 in iijj}) is $-2a^2-a+a_2+b_2$, and the coefficient of $X_2^2$ is $-a^2+b_2'+a_2'$. 
    So we have 
  \begin{align}\label{formula a}
   a_2+b_2=2a^2+a~\text{and} ~a_2'+b_2'=a^2.   
  \end{align}  
     Moreover, from the equation $c_3(p^*E)=c_1(E_1)c_2(E_2)+c_2(E_1)c_1(E_2)$, there is an equation 
    \begin{align}
        \mu_3X_1^3=&a(X_1+X_2)[(a+b_2'+1)X_1^2+(a+b_2)X_1X_2+b_2'X_2^2] \nonumber\\
        +&[(-a-2)X_1-aX_2][a_2'X_1^2+a_2X_1X_2+a_2'X_2^2].\label{formula 2 in iijj}
    \end{align}
    By comparing the coefficients of $X_1^2X_2,~X_1X_2^2$ and $X_2^3$ in (\ref{formula 2 in iijj}), we have the following equations
    \begin{align}
        2a^2-aa_2-aa_2'+ab_2+ab_2'+a-2a_2=0,\label{formula 3 in iijj}\\
        a^2-aa_2-aa_2'+ab_2+ab_2'-2a_2'=0,\label{formula 4 in iijj}\\
        a(b_2'-a_2')=0.\label{formula 5 in iijj}
    \end{align}
    If $a$ is $0$, by (\ref{formula a}), (\ref{formula 3 in iijj}) and (\ref{formula 4 in iijj}), we have $a_2=a_2'=0$ and $b=b_2=b_2'=0$. Then $E$ splits by similar arguments as in \textbf{Case I}.\\
    
    If $a$ is not $0$, we have $a_2'=b_2'=\frac{a^2}{2}$ by (\ref{formula 5 in iijj}) and (\ref{formula a}). From (\ref{formula 4 in iijj}) and (\ref{formula a}), we get $a_2=b_2=a^2+\frac{a}{2}$. By these relations and the equation $c_4(p^*E)=c_2(E_1)c_2(E_2)$, we get
    \begin{align}
        \mu_4X_1^4&=[\frac{a^2}{2}X_1^2+(a^2+\frac{a}{2})X_1X_2+\frac{a^2}{2}X_2^2][(\frac{a^2}{2}+a+1)X_1^2+(a^2+\frac{3a}{2})X_1X_2+\frac{a^2}{2}X_2^2].\label{formula 6 in iijj}
    \end{align}
     The coefficient of $X_1X_2^3$ in (\ref{formula 6 in iijj}) is $\frac{a^2}{2}(a^2+\frac{a}{2})+\frac{a^2}{2}(a^2+\frac{3a}{2})=a^3(a+1)$. By the equality $c_4(p^*E)= \mu_4X_1^4 ~(\mathrm{mod}~X_1^4+X_1^2X_2^2+X_2^4)$, we deduce $a=-b=-1$ and $a_2=b_2=a_2'=b_2'=\frac{1}{2}$. So the Chern classes of $E_1$ and $E_2$ are as follows
     \begin{equation*}
        \begin{aligned}
            &c_1(E_1)=-X_1-X_2,~c_2(E_1)=\frac{1}{2}(X_1^2+X_2^2+X_1X_2),\\
            &c_1(E_2)=-X_1+X_2,~c_2(E_2)=\frac{1}{2}(X_1^2+X_2^2-X_1X_2).
        \end{aligned}
    \end{equation*}
    Hence $\det(E_2|_{p^{-1}(x)})$ is isomorphic to $\oh_{\mathbb{Q}^3}(1)$ for each $x\in X$.\\

    There is a morphism $\psi_x:p^{-1}(x)\rightarrow Gr(2,4)$  induced by the relative H-N filtration. Let $U$ be the universal subbundle and let $Q$ be the universal quotient bundle on $Gr(2,4)$. 
    There are isomorphisms 
    $\psi_x^*\oh_{\mathbb{Q}^4}(1)\cong \psi_x^*(\det(Q))\cong \det(E_2|_{p^{-1}(x)})\cong \oh_{\mathbb{Q}^3}(1)$ (here we use the isomorphism $\mathbb{Q}^4\cong Gr(2,4)$). So $\psi_x$ maps lines on $\mathbb{Q}^3$ isomorphically to lines on $\mathbb{Q}^4$. Therefore, as the pull-back of uniform bundles on $\mathbb{Q}^4$, $E_1|_{p^{-1}(x)}$ and $E_2|_{p^{-1}(x)}$ are uniform bundles on $\mathbb{Q}^3$.\\
    
    Let $S$ be the spinor bundle on $\mathbb{Q}^3$. We note that the Chern classes of $E_1|_{p^{-1}(x)}$ (resp. $E_2|_{p^{-1}(x)}$) are the same as those of $S$ (resp. $S(1)$). By the classification of uniform bundles of rank $2$ on $\mathbb{Q}^3$ (cf. \cite[Theorem 3.1]{fritzsche1983linear}), we have $E_2|_{p^{-1}(x)}\cong S(1)$ and $E_1|_{p^{-1}(x)}\cong S$. By the base change theorem and the vanishings $H^i(\mathbb{Q}^3,S)=0(0\leq i \leq 3)$ (cf. \cite[Theorem 2.3]{ottaviani1988spinor}), we get $E\cong p_*E_2$. Then similar arguments as in Theorem \ref{Sharp bundle on X VMRT Grassmannnian} apply.
    \end{proof}

\subsection{The cases that the VMRT of $X$ is a product of several factors}
\ 
\newline
\indent We introduce some notations as follows. Let $X(=G/P_{\beta})$ be a generalised Grassmannian  where $\beta$ is a long root. Let $N$ be the set of roots that are adjacent to $\beta$. We assume that $N$ consists of at least two roots. Now we choose and fix a root $\beta_1\in N$ and set $N_1$ to be $N\backslash \{\beta_1\}$. \\

Let $G_1$ be the subgroup of $G$, where $\mathcal{D}(G_1)$ is the connected component of $\mathcal{D}(G)\backslash N_1$ containing $\beta$. Let $G_1'$ be the subgroup of $G_1$, where $\mathcal{D}(G_1')$ is the connected component of $\mathcal{D}(G_1)\backslash \{\beta\}$ containing $\beta_1$. Let $G_2$ be the subgroup of $G$, where $\mathcal{D}(G_2)$ is the connected component of $\mathcal{D}(G)\backslash \{\beta_1\}$ containing $\beta$.\\

We now consider the Tits fibration:
\[
    \begin{tikzcd}
        G/P_{\beta,\beta_1}\arrow[r,"q_1"]\arrow[d,"p_1"]& G/P_{\beta_1}\\
        G/P_\beta,
    \end{tikzcd}
 \]
 then $q_1^{-1}(y)$ is isomorphic to $G_2/P_\beta$ for any $y\in G/P_{\beta_1}$, $p_1^{-1}(x)$ is isomorphic to $G_1'/P_{\beta_1}$ for any $x\in G/P_\beta$.\\

We have a proposition as follows. 
\begin{proposition}\label{subfamily to uniqueness}
   Assume that $G_1/P_\beta$ is $\mathbb{P}^k(k>1)$ or $\mathbb{Q}^n(n=3,5,6)$. Let $E$ be a uniform bundle of rank $e.d.(\mathrm{VMRT}_{G_1/P_\beta})+1$ on $G/P_\beta$.  Let $\tau_1:G_1/P_\beta \hookrightarrow G/P_\beta \text{ and } \tau_2:G_2/P_\beta \hookrightarrow G/P_\beta$ be the sub-diagram embeddings (for the definition of sub-diagram embeddings, we refer to \cite[section 3.1]{MR2785842}). \\
   
   Suppose that $\tau_2^*(E)$ splits for any $\tau_2$, and $\tau_1^*(E)$ is unsplit for any $\tau_1$. Then there exists a bundle $E_0$ (independent of $E$) such that $E$ is isomorphic to $E_0\otimes L$ or $E_0^{\vee}\otimes L$ for some line bundle $L$.
\end{proposition}
\begin{proof}
   The morphism $p_1$ maps $q_1^{-1}(y)$ isomorphically to its image and the image is a sub-diagram embedding $\tau_2:G_2/P_\beta \hookrightarrow G/P_\beta$. By the assumption, $\tau_2^*(E)$ splits for any $\tau_2$, we have $p_1^*E|_{q_1^{-1}(y)}$ splits for any $y\in G/P_{\beta_1}$. According to the relative H-N filtration of $p_1^*E$, there exists an exact sequence of vector bundles $0\rightarrow E_1 \rightarrow p_1^*E \rightarrow E_2 \rightarrow 0$, where $E_2|_{q_1^{-1}(y)}$ is the last H-N factor of $p_1^*E|_{q_1^{-1}(y)}$.\\

   We consider the Tits fibration 
   \[
    \begin{tikzcd}
        G_1/P_{\beta,\beta_1}\arrow[r,"q"]\arrow[d,"p"]& G_1/P_{\beta_1}\\
        G_1/P_\beta,
    \end{tikzcd}
 \]
   where $G_1/P_{\beta_1}$ is the Fano scheme of lines on $G_1/P_\beta$. We denote $\tau:G_1/P_{\beta,\beta_1}\hookrightarrow G/P_{\beta,\beta_1}$ the sub-diagram embedding. Note that for any $y\in G_1/P_{\beta_1}$, $q^{-1}(y)$ is a line in $q_1^{-1}(y)$. So the restriction of $\tau^*p_1^*E(=p^*\tau_1^*E)$ to $q^{-1}(y)$ has the same splitting type of $p_1^*E|_{q_1^{-1}(y)}$. So there exists an exact sequence $0\rightarrow \tau^*E_1 \rightarrow p^*\tau_1^*E \rightarrow \tau^*E_2 \rightarrow 0$ where $\tau^*E_2|_{q^{-1}(y)}$ is the last H-N factor of $p^*\tau_1^*E|_{q^{-1}(y)}$.\\

   By the assumptions that $G_1/P_\beta$ is $\mathbb{P}^k(k>1)$ or $\mathbb{Q}^n(n=3,5,6)$ and the classification of unsplit uniform bundles of minimal ranks on $G_1/P_\beta$, up to twisting $E$ by a line bundle or taking the dual sheaf, $\tau^*E_1|_{p^{-1}(x)}$ and $\tau^*E_2|_{p^{-1}(x)}$ are fixed bundles on $p^{-1}(x)$, which are independent of $x$.
   We denote $\tau^*E_1|_{p^{-1}(x)}$ by $F_1$ and denote $\tau^*E_2|_{p^{-1}(x)}$ by $F_2$. We may assume the vanishings of $H^i(p^{-1}(x),F_1)(i=0,1)$ and assume that $F_2$ is simple (see \cite[Section 3.5 and Section 4.5 Proposition A]{EHS}, \cite[Page 339]{fritzsche1983linear} and the proofs of Theorem \ref{Sharp bundle on X VMRT Grassmannnian} and Theorem \ref{sharp bundle on Q5}). \\

   As the sub-diagram embedding $\tau_1$ varies, the image of $\tau_1$ covers $G/P_{\beta}$. For $x\in G/P_{\beta}$, we may assume that $x
   $ is in some $\tau_1(G_1/P_\beta)$. Note the equality $p^{-1}(x)=p_1^{-1}(x)(\cong G_1'/P_{\beta_1})$. There are isomorphisms $E_1|_{p_1^{-1}(x)}\cong \tau^*E_1|_{p^{-1}(x)}\cong  F_1$ and $E_2|_{p_1^{-1}(x)}\cong \tau^*E_2|_{p^{-1}(x)}\cong F_2$. Then by similar arguments as in Theorem \ref{Sharp bundle on X VMRT Grassmannnian}, $E$ is ${p_1}_*(E_2)$ and our assertion follows.
\end{proof}

For generalised Grassmannians in the following Table \ref{Table 4},  we choose some suitable $\beta_1$ such that $G_1/P_\beta$ is a projective space or a quadric, where $\alpha_k$ is the $k-$th node in $\mathcal{D}(G)$. 
    \begin{center}
        \begin{longtable}{|c|c|c|c|c|c|}
        \hline
            $X$ & $\beta_1$ & $G_1/P_\beta$ & $G_2/P_\beta$ & $e.d.(\mathrm{VMRT}_{G_1/P_\beta})$ & $e.d.(\mathrm{VMRT}_{G_2/P_\beta})$\\
        \hline
            $B_n/P_k(2\leq k <\frac{2n}{3})$ & $\alpha_{k-1}$ & $\mathbb{P}^k$ & $\mathbb{Q}^{2(n-k)+1}$ & $k-1$ & $2n-2k-1$\\
        \hline
            \multirow{2}{*}{$B_6/P_4$} & $\alpha_3$ & $\mathbb{P}^4$ & $\mathbb{Q}^5$ & $3$ & $3$\\
        \cline{2-6} 
             & $\alpha_5$ & $\mathbb{Q}^5$ & $\mathbb{P}^4$ & $3$ & $3$\\
        \hline 
            $B_n/P_{n-2}(n>6)$ & $\alpha_{n-1}$ & $\mathbb{Q}^5$ & $\mathbb{P}^{n-2}$ & $3$ & $n-3$ \\
        \hline
            $B_n/P_{n-1}(n>3)$ & $\alpha_{n}$ & $\mathbb{Q}^3$ & $\mathbb{P}^{n-1}$ & $1$ & $n-2$ \\
        \hline 
            $D_n/P_k(2\leq k <\frac{2n-2}{3})$ & $\alpha_{k-1}$ & $\mathbb{P}^k$ & $\mathbb{Q}^{2(n-k)}$ & $k-1$ & $2n-2k-3$\\
        \hline
            \multirow{2}{*}{$D_7/P_4$} & $\alpha_3$ & $\mathbb{P}^4$ & $\mathbb{Q}^6$ & $3$ & $3$\\
        \cline{2-6} 
             & $\alpha_5$ & $\mathbb{Q}^6$ & $\mathbb{P}^4$ & $3$ & $3$\\
        \hline
            $D_n/P_{n-3}(n>7)$ & $\alpha_{n-2}$ & $\mathbb{Q}^6$ & $\mathbb{P}^{n-3}$ & $3$ & $n-4$\\
        \hline
            $E_n/P_3(n=6,7,8)$ & $\alpha_1$ & $\mathbb{P}^2$ & $OG(n-1,2n-2)$ & $1$ & $n-2$\\
        \hline 
            $E_n/P_4(n=6,7,8)$ & $\alpha_2$ & $\mathbb{P}^2$ & $Gr(3,n)$ & $1$ & $2$\\
        \hline
            $E_n/P_5(n=6,7,8)$ & $\alpha_6$ & $\mathbb{P}^{n-4}$ & $OG(5,10)$ & $n-5$ & $4$\\
        \hline 
            $E_n/P_6(n=7,8)$ & $\alpha_7$ & $\mathbb{P}^{n-5}$ & $E_6/P_6$ & $n-6$ & $7$\\
        \hline
            $E_8/P_7$ & $\alpha_8$ & $\mathbb{P}^2$ & $E_7/P_7$ & $1$ & $12$\\
        \hline 
            $F_4/P_2$ & $\alpha_1$ & $\mathbb{P}^2$ & $LG(3,6)$ & $1$ & $2$\\
        \hline
        \caption{}
    \label{Table 4}
        \end{longtable}
    \end{center}

\begin{theorem}\label{sharp bundle on ed G1/P<ed G2/P}
     For generalised Grassmannians in Table \ref{Table 4}, the statement $(3)$ in Theorem \ref{main1} is true.
\end{theorem}

\begin{proof}
    We will keep using the notations in the proof of Proposition \ref{subfamily to uniqueness}. Let $X$ be a generalised Grassmannian in Table \ref{Table 4} which is not $B_n/P_{n-1}(n>3),~F_4/P_2,~B_6/P_4$ or $D_7/P_4$. 
   We have $e.d.(\mathrm{VMRT}_X)=e.d.(\mathrm{VMRT}_{G_1/P_\beta})$ and $e.d.(\mathrm{VMRT}_{G_1/P_\beta})<e.d.(\mathrm{VMRT}_{G_2/P_\beta})$. Any $2$-plane in $X$ is contained in a $G_1/P_\beta$ or $G_2/P_\beta$. As $\mathrm{rk}(E)(=e.d.(\mathrm{VMRT}_X)+1)$ is at most $e.d.(\mathrm{VMRT}_{G_2/P_\beta})$, $E|_{G_2/P_\beta}$ splits. So by \cite[Corollary 3.6]{du2021vector}, $E|_{G_1/P_\beta}$ does not split for at least one sub-diagram embedding. As the uniform bundles of rank $e.d.(\mathrm{VMRT}_{G_1/P_\beta})+1$ on $G_1/P_\beta$ are determined by their Chern classes (see \cite{EHS} or \cite[Theorem 2]{guyot1985caracterisation}, Theorem \ref{Sharp bundle on X VMRT Grassmannnian} and Theorem \ref{sharp bundle on Q5}), $E|_{G_1/P_\beta}$ does not split for any sub-diagram embedding $G_1/P_\beta \hookrightarrow G/P_\beta$. Then our assertion follows from Proposition \ref{subfamily to uniqueness}. \\

    If $X$ is $B_n/P_{n-1}(n>3)$ or $F_4/P_2$, the rank of $E$ is $2$ and hence $E|_{G_2/P_\beta}$ splits. Now we assume that  $E|_{G_1/P_\beta}$ splits. Note that in this case, $G_1/P_\beta$ is $\mathbb{Q}^3$ or $\mathbb{P}^2$. By Proposition \ref{only zeros and ones in the type}, without loss of generality, we may assume the splitting type of $E$ is $(0,-1)$. Then $E_1$ and $E_2$ in the exact sequence $0\rightarrow E_1 \rightarrow p_1^*E \rightarrow E_2 \rightarrow 0$ are line bundles. Since $E|_{G_1/P_\beta}$ splits, according to the proof of the classification of unsplit uniform bundles on $\mathbb{Q}^3$ and $\mathbb{P}^2$, we know that $\tau_1^*E_1|_{p^{-1}(x)}$ and $\tau_1^*E_2|_{p^{-1}(x)}$ are trivial line bundles, which implies that $E_1|_{p_1^{-1}(x)}$ and $E_2|_{p_1^{-1}(x)}$ are trivial bundles. Then there is an exact sequence $0\rightarrow {p_1}_*E_1 \rightarrow E \rightarrow {p_1}_*E_2 \rightarrow 0$, where ${p_1}_*E_1$ and ${p_1}_*E_2$ are line bundles. As $\dim(X)$ is at least $2$, the above sequence splits and $E$ is isomorphic to ${p_1}_*E_1\oplus {p_1}_*E_2$. Therefore if $E|_{G_1/P_\beta}$ splits, $E$ splits. So we may assume $E|_{G_1/P_\beta}$ unsplits.  Then by Proposition \ref{subfamily to uniqueness}, we are done.\\

    If $X$ is $B_6/P_4$, the rank of $E$ is $4$. Every $2$-plane is contained in a $\mathbb{P}^4$ or a $\mathbb{Q}^5$. If $E|_{\mathbb{P}^4}$ and $E|_{\mathbb{Q}^5}$ split, then by \cite[Corollary 3.6]{du2021vector}, $E$ splits. If $E|_{\mathbb{P}^4}$ splits and $E|_{\mathbb{Q}^5}$ is unsplit, we choose $\beta_1$ to be the fifth root in $B_6$. If $E|_{\mathbb{Q}^5}$ splits and $E|_{\mathbb{P}^4}$ is unsplit, we choose $\beta_1$ to be the third root in $B_6$. Note that the splitting types of unsplit uniform bundles of rank $4$ on $\mathbb{P}^4$ and $\mathbb{Q}^5$ are different. Our theorem follows from Proposition \ref{subfamily to uniqueness}. The proof for the case that $X$ is $D_7/P_4$ is quite similar. 
\end{proof}

Now we deal with the remaning cases in Theorem \ref{main1}, $(3)$.

\begin{proposition}\label{sharp bundle on B3/P2}
    Let $X$ be the generalised Grassmannian $B_3/P_{2}$. Then Theorem \ref{main1}, $(3)$ is true for $X.$
\end{proposition}
\begin{proof}
    The universal family of lines $\mathcal{U}$ is $B_3/P_{1,2,3}$ and the parameter space $\mathcal{M}$ is $B_3/P_{1,3}$. By Lemma \ref{q-cohomology of Bn/P}, we list the cohomology rings as follows:
    \begin{equation*}
        \begin{aligned}
            &H^{\bullet}(X,\mathbb{Q})=\mathbb{Q}[X_1,X_2;]/(\Sigma_i(X_1^2,X_{2}^2)_{2\leq i \leq 3});\\
            &H^{\bullet}(\mathcal{U},\mathbb{Q})=\mathbb{Q}[X_1;X_{2};X_3;]/(\Sigma_i(X_1^2,X_2^2,X_{3}^2)_{1\leq i \leq 3});\\
            &H^{\bullet}(\mathcal{M},\mathbb{Q})=\mathbb{Q}[X_1;X_{2},X_{3};]/(\Sigma_i(X_1^2,X_2^2,X_{3}^2)_{1\leq i \leq 3}).
        \end{aligned}
    \end{equation*}
    By Proposition \ref{only zeros and ones in the type}, we may assume that the splitting type of $E$ is $(0,-1)$. Then the relative H-N filtration of $p^*E$ induces an exact sequence:
    \begin{align}\label{exact for Bn/Pn-1}
        0\rightarrow E_1\rightarrow p^*E \rightarrow E_2\rightarrow 0.
    \end{align}
    There are line bundles $G_1,G_2$ on $\mathcal{M}$ satisfying $E_1=q^*G_1$ and $E_2=q^*G_2\otimes p^*\oh_X(-1)$. We denote the Chern characters by 
    $c_1(E_1)=aX_1+b(X_2+X_{3}),~c_1(E_2)=a'X_1+b'(X_2+X_3)-(X_1+X_2),~c_i(p^*E)=s_i(X_1,X_{2})(i=1,2)$, 
    where all undetermined coefficients are rational numbers. The homogeneous polynomial $s_i(X_1,X_2)$ is of degree $i$ and is symmetrical in variables $X_1,X_{2}$.\\

    Note that the degrees of the generators in the ideals of cohomology rings are at least $2$. From the equation $c_1(p^*E)=c_1(E_1)+c_1(E_2)$, we get
    $(a+a')(X_1)+(b+b')(X_2+X_{3})-(X_1+X_{2})=s_1(X_1,X_{2}).$
    The coefficient of $X_3$ in $s_1(X_1,X_2)$ is $0$, which implies that $b+b'$ is $0$. Since $s_1(X_1,X_2)$ is symmetrical in $X_1$ and $X_2$, $a+a'$ is $0$. From $c_2(p^*E)=c_1(E_1)c_1(E_2)$, we have
    \begin{align}
        &-[aX_1+bX_{2}+bX_3][(a+1)X_1+(b+1)X_{2}+bX_3]=s_2(X_1,X_{2})+\mu(X_1^2+X_2^2+X_3^2),\label{Bn/Pn-1 c_1(E1)c_1(E_2)}
    \end{align}
    where $\mu$ is a rational number. The coefficient of $X_1X_3$ on the left side of (\ref{Bn/Pn-1 c_1(E1)c_1(E_2)}) is $-b(2a+1)$ and the coefficient of $X_2X_{3}$ on the left side of (\ref{Bn/Pn-1 c_1(E1)c_1(E_2)}) is $-b(2b+1)$. So both $b(2a+1)$ and $b(2b+1)$ are $0$.\\

    If $b$ is $0$, the coefficient of $X_3^2$ on the left side of (\ref{Bn/Pn-1 c_1(E1)c_1(E_2)}) is $0$. So $\mu$ is $0$ and we have
    $-aX_1[(a+1)X_1+X_{2}]=s_2(X_1,X_{2})$. 
    As $s_2(X_1,X_2)$ is symmetrical in $X_1$ and $X_2$, the only possibilities are $a=0$ or $a=-1$. \\
    
    When $a$ is $0$, both the line bundles $G_1$ and $G_2$ are trivial. By performing $R^ip_*$ to the exact sequence (\ref{exact for Bn/Pn-1}), we deduce that $E$ splits, which is impossible. When $a$ is $-1$, 
    $c_1(E_1)$ is $-X_1$ and $c_1(E_2)$ is $-X_2$$(=X_1-(X_1+X_2))$. 
    Let $L_{\lambda_1}$ be the line bundle on $\mathcal{U}=B_3/P_{1,2,3}$ corresponding to the weight $\lambda_1$. We have $E_1\cong L_{\lambda_1}^{-1}$ and $E_2\cong L_{\lambda_1}\otimes p^*\oh_X(-1)$. For each fiber $p^{-1}(x)(=\mathbb{P}^{1}\times \mathbb{P}^1)$, $E_1|_{p^{-1}(x)}$ is isomorphic to $\oh_{\mathbb{P}^{1}}(-1) \boxtimes \oh_{\mathbb{P}^1}$. By the base change theorem,  we have $p_*E_1=R^1p_*E_1=0$ . Hence there are isomorphisms $E\cong p_*L_{\lambda_1}\otimes \oh_X(-1)\cong E_{\lambda_1}\otimes \oh_X(-1)$.\\

    If $b$ is not $0$, then we have $a=b=-\frac{1}{2}$. The Chern classes of $E_1$ and $E_2$ satisfy
    $$c_1(E_1)=-\frac{1}{2}(X_1+X_2+X_3)~\text{and}~c_1(E_2)=\frac{1}{2}(X_1+X_2+X_3)-(X_1+X_{2}).$$
    Let $L_{\lambda_3}$ be the line bundle on $\mathcal{U}=B_3/P_{1,2,3}$ corresponding to the weight $\lambda_3$. By comparing the first Chern classes of $E_1$ and $E_2$ with the class of $L_{\lambda_3}$ (see \cite[PLATE II (VI)]{lieGroupsandLiealgebras4-6}),
    we have $E_1\cong L_{\lambda_3}^{-1}$ and $E_2\cong L_{\lambda_3}\otimes p^*\oh_X(-1)$. For each fiber $p^{-1}(x)=\mathbb{P}^{1}\times \mathbb{P}^1$, $E_1|_{p^{-1}(x)}$ is isomorphic to $\oh_{\mathbb{P}^{1}} \boxtimes \oh_{\mathbb{P}^1}(-1)$. By the base change theorem, we have $p_*E_1=R^1p_*E_1=0$ . Therefore $E$ is $E_{\lambda_3}\otimes \oh_X(-1)$.
\end{proof}

\begin{proposition}\label{sharp bundle on Dn/Pn-2}
    Let $X$ be the generalised Grassmannian $D_n/P_{n-2}(n\geq 4)$. Then Theorem \ref{main1}, $(3)$ is true for $X.$
    
\end{proposition}
\begin{proof}
    The universal family of lines $\mathcal{U}$ is $D_n/P_{n-3,n-2,n-1,n}$ and the parameter space $\mathcal{M}$ is $D_n/P_{n-3,n-1,n}$. 
    By Lemma \ref{q-cohomology of Dn/P}, we list the cohomology rings as follows:
    \begin{equation*}
        \begin{aligned}
            &H^{\bullet}(X,\mathbb{Q})=\mathbb{Q}[X_1,\dots,X_{n-2};\eta_{n-2}]/(\Sigma_i(X_1^2,\dots,X_{n-2}^2)_{3\leq i \leq n},\eta_{n-2}^2-\Sigma_2(X_1^2,\dots,X_{n-2}^2),X_1\cdots X_n);\\
            &H^{\bullet}(\mathcal{U},\mathbb{Q})=\mathbb{Q}[X_1,\dots,X_{n-3};X_{n-2};X_{n-1};X_n;]/(\Sigma_i(X_1^2,\dots,X_{n}^2)_{1\leq i \leq n},X_1\cdots X_n);\\
            &H^{\bullet}(\mathcal{M},\mathbb{Q})=\mathbb{Q}[X_1,\dots,X_{n-3};X_{n-2},X_{n-1};X_{n};]/(\Sigma_i(X_1^2,\dots,X_{n}^2)_{1\leq i \leq n},X_1\cdots X_n),
        \end{aligned}
    \end{equation*}
    where $\eta_{n-2}$ is $X_{n-1}X_n$.
    By Proposition \ref{only zeros and ones in the type}, up to twisting $E$ by a line bundle, we may assume that the splitting type of $E$ is $(0,-1)$. Then the relative H-N filtration of $p^*E$ induces an exact sequence
    \begin{align}\label{exact for Dn/Pn-2}
        0\rightarrow E_1\rightarrow p^*E \rightarrow E_2\rightarrow 0.
    \end{align}
   There are line bundles $G_1,G_2$ on $\mathcal{M}$ satisfying $E_1=q^*G_1$ and $E_2=q^*G_2\otimes p^*\oh_X(-1)$. We denote the Chern classes by
    \begin{equation*}
        \begin{aligned}
            &c_1(E_1)=a(X_1+\dots+X_{n-3})+b(X_{n-2}+X_{n-1})+cX_n;\\
            &c_1(E_2)=a'(X_1+\dots+X_{n-3})+b'(X_{n-2}+X_{n-1})+c'X_n-(X_1+\dots+X_{n-2});\\
            &c_1(p^*E)=s_1(X_1,\dots,X_{n-2}),~c_2(p^*E)=s_2(X_1,\dots,X_{n-2})+\gamma X_{n-1}X_n,
        \end{aligned}
    \end{equation*}
    where all undetermined coefficients are rational numbers. The homogeneous polynomial $s_i(X_1,\dots,X_{n-2})$ is of degree $i$ and is symmetrical in variables $X_1,\dots,X_{n-2}$.\\

    The degrees of the generators in the ideals of cohomology rings are at least $2$. From the equation    $c_1(p^*E)=c_1(E_1)+c_1(E_2)$, we get\
    \begin{align}
        &(a+a')(X_1+\dots+X_{n-3})+(b+b')(X_{n-2}+X_{n-1})+(c+c')X_n\nonumber\\
            =&s_1(X_1,\dots,X_{n-2})+(X_1+\dots+X_{n-2})\label{Dn/Pn-2, c1(E1)+c1(E2)}.
    \end{align}
    The coefficients of $X_{n-1}$ and $X_n$ on the right side of  (\ref{Dn/Pn-2, c1(E1)+c1(E2)}) are $0$. Both $b+b'$ and $c+c'$ are $0$. The polynomial on the right side of  (\ref{Dn/Pn-2, c1(E1)+c1(E2)}) is symmetrical in $X_1,\dots,X_{n-2}$, so we get $a+a'=0$. From the equation $c_2(p^*E)=c_1(E_1)c_1(E_2)$, we have
    \begin{align}
        &-[a(X_1+\dots+X_{n-3})+bX_{n-2}+bX_{n-1}+cX_n][(a+1)(X_1+\dots+X_{n-3})+(b+1)X_{n-2}\nonumber\\
          &+bX_{n-1}+cX_n]=s_2(X_1,\dots,X_{n-2})+\gamma X_{n-1}X_n+\mu(X_1^2+\dots+X_n^2),\label{Dn/Pn-2 c_1(E1)c_1(E_2)}      
    \end{align}
    where $\mu$ is a rational number. The coefficients of $X_jX_n(1\leq j \leq n-3)$ and $X_{n-2}X_n$ on the left side of (\ref{Dn/Pn-2 c_1(E1)c_1(E_2)}) are $-c(2a+1)$ and $-c(2b+1)$ respectively. So we have $c(2a+1)=c(2b+1)=0$.\\

    If $c$ is $0$, the coefficient of $X_n^2$ on the left side of (\ref{Dn/Pn-2 c_1(E1)c_1(E_2)}) is $0$.  So we have $\mu=0$ and the coefficient of $X_{n-1}^2$ on the right side of (\ref{Dn/Pn-2 c_1(E1)c_1(E_2)}) is $0$, which means that $b$ is $0$. Then (\ref{Dn/Pn-2 c_1(E1)c_1(E_2)}) is independent of $X_{n-1}$, we deduce $\gamma=0$ and
    $$-[a(X_1+\dots+X_{n-3})][(a+1)(X_1+\dots+X_{n-3})+X_{n-2}]=s_2(X_1,\dots,X_{n-2}).$$
    The coefficient of $X_{n-2}^2$ on the left hand side is $0$. So the coefficient of  $X_j^2(1\leq j \leq n-3)$ on the left hand side is also $0$, as  $s_2$ is a symmetric polynomial. So $a(a+1)$ is $0$.\\

    When $a$ is $0$, both the line bundles $G_1$ and $G_2$ are trivial. By performing $R^ip_*$ to the  exact sequence (\ref{exact for Dn/Pn-2}), we deduce that $E$ splits.  when $a$ is $-1$,  the product $(X_1+\dots+X_{n-3})X_{n-2}$ is symmetrical in $X_1,\dots,X_{n-2}$, which means that  $n$ is $4$. If follows that
    $c_1(E_1)$ is $-X_1$ and $c_1(E_2)$ is  $-X_2$$(=X_1-(X_1+X_2))$. Let $L_{\lambda_1}$ be the line bundle on $\mathcal{U}=D_4/P_{1,2,3,4}$ corresponding to the weight $\lambda_1$. By comparing the first Chern class with the class of $L_{\lambda_1}$ (see \cite[PLATE IV (VI)]{lieGroupsandLiealgebras4-6}), we have $E_1\cong L_{\lambda_1}^{-1}$ and $E_2\cong L_{\lambda_1}\otimes p^*\oh_X(-1)$. For each fiber $p^{-1}(x)=\mathbb{P}^{1}\times\mathbb{P}^1\times \mathbb{P}^1$, $E_1|_{p^{-1}(x)}$ is isomorphic to $\oh_{\mathbb{P}^{1}}(-1) \boxtimes \oh_{\mathbb{P}^1}\boxtimes \oh_{\mathbb{P}^1}$.  By the base change theorem, we have $p_*E_1=R^1p_*E_1=0$. Therefore $E$ is 
    $E_{\lambda_1}\otimes \oh_X(-1)$.\\

    If $c$ is not $0$, then we have $a=b=-\frac{1}{2}$. We rewrite formula (\ref{Dn/Pn-2 c_1(E1)c_1(E_2)}) into the following equations:
    \begin{align}
        &-[-\frac{1}{2}(X_1+\dots+X_{n-2})-\frac{1}{2}X_{n-1}+cX_n][\frac{1}{2}(X_1+\dots+X_{n-2})-\frac{1}{2}X_{n-1}+cX_n]\nonumber\\
            =&\frac{1}{4}(X_1+\dots+X_{n-2})^2-(cX_n-\frac{1}{2}X_{n-1})^2=s_2(X_1,\dots,X_{n-2})+\gamma X_{n-1}X_n+\mu(X_1^2+\dots+X_n^2).\label{Dn/Pn-2 coeff xn2 and Xn-12}
    \end{align}
    The coefficients of $X_{n-1}^2$ and $X_n^2$ are the same in (\ref{Dn/Pn-2 coeff xn2 and Xn-12}), so we have $c^2=\frac{1}{4}$.\\

    If $c$ is $\frac{1}{2}$, the Chern classes of $E_1$ and $E_2$ are
    $c_1(E_1)=-\frac{1}{2}(X_1+\dots+X_{n-1}-X_n)$ and $c_1(E_2)=\frac{1}{2}(X_1+\dots+X_{n-1}-X_n)-(X_1+\dots+X_{n-2}).$
    Let $L_{\lambda_{n-1}}$ be the line bundle on $\mathcal{U}=D_n/P_{n-3,n-2,n-1,n}$ corresponding to the weight $\lambda_{n-1}$. By comparing the first Chern class with the class of $L_{\lambda_{n-1}}$(see \cite[PLATE IV (VI)]{lieGroupsandLiealgebras4-6}), we have $E_1\cong L_{\lambda_{n-1}}^{-1}$ and $E_2\cong L_{\lambda_{n-1}}\otimes p^*\oh_X(-1)$. For each fiber $p^{-1}(x)=\mathbb{P}^{n-3}\times \mathbb{P}^1\times \mathbb{P}^1$, where the first $\mathbb{P}^1$ factor corresponds to the $(n-1)$-th node and the second $\mathbb{P}^1$ factor corresponds to the $n$-th node. The restriction $E_1|_{p^{-1}(x)}$ is isomorphic to $\oh_{\mathbb{P}^{n-3}} \boxtimes\oh_{\mathbb{P}^1}(-1) \boxtimes \oh_{\mathbb{P}^1}$. By the base change theorem, we have $p_*E_1=R^1p_*E_1=0$ . There are isomorphisms $E\cong p_*L_{\lambda_{n-1}}\otimes \oh_X(-1)\cong E_{\lambda_{n-1}}\otimes \oh_X(-1)$.\\
    
    If $c$ is $-\frac{1}{2}$, the Chern classes of $E_1$ and $E_2$ are
    $c_1(E_1)=-\frac{1}{2}(X_1+\dots+X_n) $and $ c_1(E_2)=\frac{1}{2}(X_1+\dots+X_n)-(X_1+\dots+X_{n-2}).$
    Let $L_{\lambda_n}$ be the line bundle on $\mathcal{U}=D_n/P_{n-3,n-2,n-1,n}$ corresponding to the weight $\lambda_n$. By comparing the first Chern class with the class of $L_{\lambda_n}$, we have $E_1\cong L_{\lambda_n}^{-1}$ and $E_2\cong L_{\lambda_n}\otimes p^*\oh_X(-1)$. By similar arguments as above, we have isomorphisms $E\cong p_*E_2\cong E_{\lambda_n}\otimes \oh_X(-1)$.
\end{proof}
Combining Theorem \ref{Sharp bundle on X VMRT Grassmannnian}, Theorem \ref{sharp bundle on Q5}, Theorem \ref{sharp bundle on ed G1/P<ed G2/P}, Proposition \ref{sharp bundle on B3/P2}, Proposition \ref{sharp bundle on Dn/Pn-2} and some known classification results (see \cite{EHS,guyot1985caracterisation, fritzsche1983linear,munoz2012uniform}), we complete the proof of Theorem \ref{main1} $(3)$.
\begin{remark}

    For the following generalised Grassmannians, we don't know whether the bounds $e.d.(\mathrm{VMRT})$ are optimal:
    \begin{equation*}
        \begin{aligned}
             &Q^n(n>6),~ B_n/P_k(\frac{2n}{3}<k<n-2), ~C_n/P_k(k<n),\\
             &D_n/P_k(\frac{2n-2}{3}<k<n-3),~E_n/P_1, E_n/P_n(n=6,7,8),~F_4/P_k(k=3,4).
        \end{aligned}
    \end{equation*}

    For the following generalised Grassmannians, the bounds $e.d.(\mathrm{VMRT})$ are optimal while the classification of unsplit uniform bundles of minimal ranks is unknown.
    \begin{equation*}
        \begin{aligned}
             B_n/P_k(k=\frac{2n}{3},n \ne  3,6),~C_n/P_n,~ ~D_n/P_k(k=\frac{2n-2}{3},n\ne 4,7),~  F_4/P_1.
        \end{aligned}
    \end{equation*}
    \\
    
    In one forthcoming paper of Xinyi Fang, she has given the classification of unsplit uniform bundles of minimal ranks on the variety $G_2/P_2$.
\end{remark}

\textbf{Funding:}
Xinyi Fang is supported by innovation Action Plan (Basic research projects) of Science and Technology Commission of Shanghai Municipality (Grant No.21JC1401900).
Duo Li is supported by National Natural Science Foundation of China (Grant No. 12001547).\\

\textbf{Acknowledgements:} The authors would like to thank Baohua Fu and Rong Du for introducing this problem to us and for their very helpful discussions during the last several months. We also thank Peng Ren for  valuable discussions and suggestions. The authors are  grateful to Gianluca Occhetta for his generosity of sharing ideas with us. His  suggestions improve the first version of this article hugely.
\bibliography{ref}
\bibliographystyle{plain}

\end{document}